\documentclass[a4paper,11pt]{amsart}

\usepackage[colorlinks,linkcolor=blue,citecolor=blue]{hyperref}
\usepackage{latexsym, amssymb, amsmath, amsthm, mathrsfs, bbm}
\usepackage[all]{xy}
\usepackage{pstricks}
\usepackage{ifthen}

\usepackage{anysize}\marginsize{22mm}{22mm}{25mm}{25mm}
\allowdisplaybreaks[4]

\newcommand{\Dchaintwo}[4]{
\rule[-3\unitlength]{0pt}{8\unitlength}
\begin{picture}(14,5)(0,3)
\put(1,2){\ifthenelse{\equal{#1}{l}}{\circle*{2}}{\circle{2}}}
\put(2,2){\line(1,0){10}}
\put(13,2){\ifthenelse{\equal{#1}{r}}{\circle*{2}}{\circle{2}}}
\put(1,5){\makebox[0pt]{\scriptsize #2}}
\put(7,4){\makebox[0pt]{\scriptsize #3}}
\put(13,5){\makebox[0pt]{\scriptsize #4}}
\end{picture}}

\newcommand{\Hexagon}[7]{
\rule[-3\unitlength]{0pt}{8\unitlength}
\begin{picture}(30,20)(0,3)
\put(1,13){\ifthenelse{\equal{#1}{l}}{\circle*{2}}{\circle{2}}}
\put(2,14){\line(1,1){6}}
\put(9,21){\ifthenelse{\equal{#1}{l}}{\circle*{2}}{\circle{2}}}
\put(10,21){\line(1,0){8}}
\put(19,21){\ifthenelse{\equal{#1}{l}}{\circle*{2}}{\circle{2}}}
\put(20,20){\line(1,-1){6}}
\put(27,13){\ifthenelse{\equal{#1}{l}}{\circle*{2}}{\circle{2}}}
\put(26,12){\line(-1,-1){6}}
\put(19,5){\ifthenelse{\equal{#1}{l}}{\circle*{2}}{\circle{2}}}
\put(18,5){\line(-1,0){8}}
\put(9,5){\ifthenelse{\equal{#1}{l}}{\circle*{2}}{\circle{2}}}
\put(8,6){\line(-1,1){6}}
\put(-3,12){\makebox[0pt]{\scriptsize #2}}
\put(9,23){\makebox[0pt]{\scriptsize #3}}
\put(19,23){\makebox[0pt]{\scriptsize #4}}
\put(31,12){\makebox[0pt]{\scriptsize #5}}
\put(19,1){\makebox[0pt]{\scriptsize #6}}
\put(9,1){\makebox[0pt]{\scriptsize #7}}
\end{picture}}

\newcommand{\THexagon}[7]{
\rule[-3\unitlength]{0pt}{8\unitlength}
\begin{picture}(30,20)(0,3)
\put(1,13){\ifthenelse{\equal{#1}{l}}{\circle*{2}}{\circle{2}}}
\put(2,14){\line(1,1){6}}
\put(9,21){\ifthenelse{\equal{#1}{l}}{\circle*{2}}{\circle{2}}}
\put(10,21){\line(1,0){8}}
\put(19,21){\ifthenelse{\equal{#1}{l}}{\circle*{2}}{\circle{2}}}
\put(20,20){\line(1,-1){6}}
\put(27,13){\ifthenelse{\equal{#1}{l}}{\circle*{2}}{\circle{2}}}
\put(26,12){\line(-1,-1){6}}
\put(19,5){\ifthenelse{\equal{#1}{l}}{\circle*{2}}{\circle{2}}}
\put(18,5){\line(-1,0){8}}
\put(9,5){\ifthenelse{\equal{#1}{l}}{\circle*{2}}{\circle{2}}}
\put(8,6){\line(-1,1){6}}
\put(2,13){\line(1,0){24}}
\put(10,21){\line(1,-2){8}}
\put(18,21){\line(-1,-2){8}}
\put(-3,12){\makebox[0pt]{\scriptsize #2}}
\put(9,23){\makebox[0pt]{\scriptsize #3}}
\put(19,23){\makebox[0pt]{\scriptsize #4}}
\put(31,12){\makebox[0pt]{\scriptsize #5}}
\put(19,1){\makebox[0pt]{\scriptsize #6}}
\put(9,1){\makebox[0pt]{\scriptsize #7}}
\end{picture}}

\newcommand{\Dchainfour}[8]{
\rule[-3\unitlength]{0pt}{5\unitlength}
\begin{picture}(38,5)(0,3)
\put(1,2){\ifthenelse{\equal{#1}{1}}{\circle*{2}}{\circle{2}}}
\put(2,2){\line(1,0){10}}
\put(13,2){\ifthenelse{\equal{#1}{2}}{\circle*{2}}{\circle{2}}}
\put(25,2){\ifthenelse{\equal{#1}{3}}{\circle*{2}}{\circle{2}}}
\put(26,2){\line(1,0){10}}
\put(37,2){\ifthenelse{\equal{#1}{4}}{\circle*{2}}{\circle{2}}}
\put(1,5){\makebox[0pt]{\scriptsize #2}}
\put(7,4){\makebox[0pt]{\scriptsize #3}}
\put(13,5){\makebox[0pt]{\scriptsize #4}}
\put(19,4){\makebox[0pt]{\scriptsize #5}}
\put(25,5){\makebox[0pt]{\scriptsize #6}}
\put(31,4){\makebox[0pt]{\scriptsize #7}}
\put(37,5){\makebox[0pt]{\scriptsize #8}}
\end{picture}}

\newcommand{\Nchainfour}[9]{
\rule[-4\unitlength]{0pt}{5\unitlength}
\begin{picture}(30,30)(0,3)
\put(2,4){\circle{2}}
\put(3,4){\line(1,0){20}}
\put(3,5){\line(1,1){20.2}}
\put(23,25){\line(-1,-1){20.2}}
\put(2,5){\line(0,1){20}}
\put(24,4){\circle{2}}
\put(24,5){\line(0,1){20}}
\put(23,5){\line(-1,1){20.2}}
\put(3,25){\line(1,-1){20.2}}
\put(24,26){\circle{2}}
\put(23,26){\line(-1,0){20}}
\put(2,26){\circle{2}}
\put(2,28){\makebox[0pt]{\scriptsize #1}}
\put(24,28){\makebox[0pt]{\scriptsize #2}}
\put(24,1){\makebox[0pt]{\scriptsize #3}}
\put(2,1){\makebox[0pt]{\scriptsize #4}}
\put(-2,15){\makebox[0pt]{\scriptsize #5}}
\put(13,27){\makebox[0pt]{\scriptsize #6}}
\put(28,15){\makebox[0pt]{\scriptsize #7}}
\put(13,2){\makebox[0pt]{\scriptsize #8}}
\put(17,22){\makebox[0pt]{\scriptsize #9}}
\put(19,12){\makebox[0pt]{\scriptsize #9}}
\end{picture}}

\newcommand{\Rchainfour}[9]{
\rule[-4\unitlength]{0pt}{5\unitlength}
\begin{picture}(30,25)(0,3)
\put(2,15){\circle{2}}
\put(3,16){\line(2,1){12}}
\put(15,22){\line(-2,-1){12.3}}
\put(3,14){\line(2,-1){12}}
\put(15,8){\line(-2,1){12.3}}
\put(17,22){\line(2,-1){12.3}}
\put(17,8){\line(2,1){12.3}}
\put(16,22){\circle{2}}
\put(16,9){\line(0,1){12}}
\put(16,8){\circle{2}}
\put(30,15){\circle{2}}
\put(2,17){\makebox[0pt]{\scriptsize #1}}
\put(16,25){\makebox[0pt]{\scriptsize #2}}
\put(30,17){\makebox[0pt]{\scriptsize #3}}
\put(16,5){\makebox[0pt]{\scriptsize #4}}
\put(9,21){\makebox[0pt]{\scriptsize #5}}
\put(23,21){\makebox[0pt]{\scriptsize #6}}
\put(24,9){\makebox[0pt]{\scriptsize #7}}
\put(8,9){\makebox[0pt]{\scriptsize #8}}
\put(18,15){\makebox[0pt]{\scriptsize #9}}
\end{picture}}

\newcommand{\Echainfour}[8]{
\rule[-4\unitlength]{0pt}{5\unitlength}
\begin{picture}(20,20)(0,3)
\put(2,4){\circle{2}}
\put(3,4){\line(1,0){10}}
\put(2,5){\line(0,1){10}}
\put(14,4){\circle{2}}
\put(14,5){\line(0,1){10}}
\put(14,16){\circle{2}}
\put(13,16){\line(-1,0){10}}
\put(2,16){\circle{2}}
\put(8,1){\makebox[0pt]{\scriptsize #1}}
\put(2,1){\makebox[0pt]{\scriptsize #2}}
\put(14,1){\makebox[0pt]{\scriptsize #3}}
\put(2,18){\makebox[0pt]{\scriptsize #4}}
\put(14,18){\makebox[0pt]{\scriptsize #5}}
\put(18,10){\makebox[0pt]{\scriptsize #6}}
\put(8,17){\makebox[0pt]{\scriptsize #7}}
\put(-2,10){\makebox[0pt]{\scriptsize #8}}
\end{picture}}

\newcommand{\Chainfour}[6]{
\rule[-4\unitlength]{0pt}{5\unitlength}
\begin{picture}(40,4)(0,3)
\put(1,1){\circle{2}}
\put(2,1){\line(1,0){10}}
\put(13,1){\circle{2}}
\put(25,1){\circle{2}}
\put(26,1){\line(1,0){10}}
\put(37,1){\circle{2}}
\put(1,4){\makebox[0pt]{\scriptsize #1}}
\put(7,3){\makebox[0pt]{\scriptsize #2}}
\put(13,4){\makebox[0pt]{\scriptsize #3}}
\put(25,4){\makebox[0pt]{\scriptsize #4}}
\put(31,3){\makebox[0pt]{\scriptsize #5}}
\put(37,4){\makebox[0pt]{\scriptsize #6}}
\end{picture}}

\newcommand{\Mchainfour}[7]{
\rule[-4\unitlength]{0pt}{5\unitlength}
\begin{picture}(40,4)(0,3)
\put(1,1){\circle{2}}
\put(2,1){\line(1,0){10}}
\put(13,1){\circle{2}}
\put(14,1){\line(1,0){10}}
\put(25,1){\circle{2}}
\put(26,1){\line(1,0){10}}
\put(37,1){\circle{2}}
\put(1,4){\makebox[0pt]{\scriptsize #1}}
\put(7,3){\makebox[0pt]{\scriptsize #2}}
\put(13,4){\makebox[0pt]{\scriptsize #3}}
\put(19,3){\makebox[0pt]{\scriptsize #4}}
\put(25,4){\makebox[0pt]{\scriptsize #5}}
\put(31,3){\makebox[0pt]{\scriptsize #6}}
\put(37,4){\makebox[0pt]{\scriptsize #7}}
\end{picture}}

\def \To{\longrightarrow}
\def \dim{\operatorname{dim}}
\def \deg{\operatorname{deg}}
\def \Deg{\operatorname{Deg}}
\def \D{\operatorname{d}}

\def \Rank{\operatorname{Rank}}
\def \GKdim{\operatorname{GKdim}}
\def \ad{\operatorname{ad}}
\def \lim{\operatorname{lim}}

\def \sup{\operatorname{sup}}
\def \log{\operatorname{log}}

\def \mod{\operatorname{mod}}

\def \N{\mathbb{N}}

\def \D{\Delta}

\def \e{\varepsilon}

\def \Z{\mathbb{Z}}

\def \k{\mathbbm{k}}
\def \1{\mathbf{1}}
\def \id{\operatorname{id}}

\numberwithin{equation}{section}

\newtheorem{theorem}{Theorem}[section]
\newtheorem{lemma}[theorem]{Lemma}
\newtheorem{proposition}[theorem]{Proposition}
\newtheorem{corollary}[theorem]{Corollary}

\newtheorem{definition}[theorem]{Definition}

\newtheorem{remark}[theorem]{Remark}

\begin{document}

\title[Finite GK-dimensional Nichols algebras]{On the classification of finite GK-dimensional pre-Nichols algebras and quasi-quantum groups}

\author{Yuping Yang}
\address{School of Mathematics and statistics, Southwest University, Chongqing 400715, China}
\email{yupingyang@swu.edu.cn}

\date{}
\maketitle

\begin{abstract} 
We prove that every pre-Nichols algebra of a nondiagonal object in the twisted Yetter-Drinfeld category ${_{\k G}^{\k G} {\mathcal{YD}^\Phi}}$ has infinite Gelfand-Kirillov dimension, where $G$ is a finite abelian group and $\Phi$ is a $3$-cocycle on $G$. This leads to a complete characterization of finite GK-dimensional Nichols algebras in this category. Specifically, for any finite-dimensional $V\in {_{\k G}^{\k G} {\mathcal{YD}^\Phi}}$, we show that the Nichols algebra $B(V)$ has finite Gelfand-Kirillov dimension if and only if it is of diagonal type and its associated root system is finite, that is, an arithmetic root system. Via bosonization, this result yields a classification of finite GK-dimensional coradically graded pointed coquasi-Hopf algebras over finite abelian groups that are generated by group-like and skew-primitive elements.
\end{abstract}

\section{introduction}

This work contributes to the classification of infinite-dimensional pointed coquasi-Hopf algebras of finite Gelfand-Kirillov dimension (abbreviated as GKdim), along with their corresponding pre-Nichols algebras.
Throughout, we work over an algebraically closed field $\k$ of characteristic zero.

Over the past two decades, the lifting method introduced by Andruskiewitsch and Schneider \cite{AS1} has led to significant advances in the classification of finite-dimensional pointed Hopf algebras, see \cite{AS2, A2,H0, H4, HMV,HV1} and the references therein. Notably, Heckenberger \cite{H0}  developed the theory of Weyl groupoids and arithmetic root systems, which has since become an essential tool in the classification program for pointed Hopf algebras. More recently,  Andruskiewitsch, Angiono, Heckenberger, and their collaborators have extended this classification effort to pointed Hopf algebras with finite GKdim over abelian groups \cite{AAH, AAH2, AC24, AI22}. A central objective in this research direction is the classification of finite GK-dimensional pre-Nichols algebras in the Yetter-Drinfeld module category $_{\k G}^{\k G} {\mathcal{YD}}$, where $G$ is an abelian group.

Recall that a pre-Nichols algebra of a braided vector space $V$ is an $\N_0$-graded connected braided Hopf algebra $\mathcal{P}=\oplus_{i\geq 0} \mathcal{P}_i$ which is generated by $\mathcal{P}_1=V$. Let $T(V)$ and $B(V)$ denote the tensor algebra and the Nichols algebra of $V$ respectively, and let $\mathcal{P}(V)$ be a pre-Nichols algebra of $V$. Then there are two canonical surjective homomorphisms of graded braided Hopf algebras: 
$$T(V)\twoheadrightarrow \mathcal{P}(V)\twoheadrightarrow B(V).$$
Thus, a natural strategy for classifying finite GK-dimensional pre-Nichols algebras in ${_{\k G}^{\k G} {\mathcal{YD}}}$ is as follows. First, classify all finite GK-dimensional Nichols algebras. Then, for each such Nichols algebra, determine all finite GK-dimensional pre-Nichols algebras that cover it. 

The study of pointed coquasi-Hopf algebras, or dually basic quasi-Hopf algebras, is motivated by the theory of pointed tensor categories \cite{EGNO}. In \cite{EG1,EG2,EG3, G}, Etingof and Gelaki started the classification of finite dimensional basic quasi-Hopf algebras. Angiono \cite{A1} classified finite dimensional basic quasi-Hopf algebras over cyclic groups with orders not divided by small prime numbers. In \cite{HLY, HLY15}, Huang, Liu and Ye presented a classification of finite dimensional basic quasi-Hopf algebras via representation type. In a series of works \cite{HLYY,HLYY2,HLYY3,HYZ}, the author of the present paper and his collaborators classified finite dimensional coradically graded pointed coquasi-Hopf algebras over abelian groups, and partially proved a generation conjecture \cite[Conjecture 5.11.10]{EGNO} of pointed finite tensor categories. However, until now there are barely any classification results of infinite dimensional pointed coquasi-Hopf algebras which are not Hopf algebras. 

In this paper, we initiate the study of infinite dimensional pointed coquasi-Hopf algebras with finite GKdim. Adopting a quasi-version of the lifting method, we focus primarily on  Nichols algebras and pre-Nichols algebras in the twisted Yetter-Drinfeld category ${_{\k G}^{\k G} {\mathcal{YD}^\Phi}}$, where $G$ is a finite abelian group and $\Phi$ is a $3$-cocycle on $G$. It is worth noting that  Angiono and Garc\'ia Iglesias \cite{AI22} proved that a Nichols algebra of diagonal type in ${_{\k G}^{\k G} {\mathcal{YD}}}$ has finite GKdim if and only if its root system is finite. Using the the reduction method of \cite{HLYY2, HLYY3}, this result can be generalized to Nichols algebras of diagonal type in ${_{\k G}^{\k G} {\mathcal{YD}^\Phi}}$. Hence, the principal objective of this paper is to treat the case of Nichols algebras of nondiagonal type.

The novelty  of this work lies primarily in the following two aspects. First, we investigate the homogenous subalgebra $A(W)$ generated by $W$ inside a graded pre-Nichols algebra $P(V)$, and establish a sufficient condition on $W$ for $A(W)$ itself to be a pre-Nichols algebra. By utilizing this construction, we derive the key inequalities: $$\GKdim(B(W))\leq \GKdim(A(W))\leq \GKdim(\mathcal{P}(V)).$$ 
Consequently, we are able to employ arithmetic root systems to study $\GKdim(B(W))$ and, via the above inequalities, to obtain estimates for $\GKdim(\mathcal{P}(V))$.
Second, we introduce the concept of minimal nondiagonal objects in ${_{\k G}^{\k G} {\mathcal{YD}^\Phi}}$ and furnish a detailed structural characterization of such objects. This desription allows us to compute the GKdim of their associated Nichols algebras, ultimately proving that every Nichols algebra of nondiagonal type in ${_{\k G}^{\k G} {\mathcal{YD}^\Phi}}$ has infinite GKdim.

Now we present the main result of this paper. The following is our first main theorem, which provides a classification of finite GK-dimensional Nichols algebras of simple objects in ${_{\k G}^{\k G} \mathcal{YD}^\Phi}$. 

\begin{theorem}[Theorem \ref{t4.6}]\label{t1.1}
Let $V\in {_{\k G}^{\k G} \mathcal{YD}^\Phi}$ be a simple twisted Yetter-Drinfeld module with $\dim(V)\geq 2$, and the $G$-degree of $V$ is $g$. Then $B(V)$ has finite GKdim if and only if $V$ is one of the following types:
\begin{itemize}
\item[(T1)] $g\triangleright v=v$ for all $v\in V.$
\item[(T2)] $g\triangleright v=-v$ for all $v\in V.$
\item[(T3)] $g\triangleright v=\zeta_3 v$ for all $v\in V$ and $\dim(V)=2$, where $\zeta_3$ is a primitive third root of unity.
\end{itemize}
\end{theorem}
If $V$ is of type (T1), then $B(V)\cong S(V)$ and $\GKdim(B(V))=\dim(V)$. For types (T2) and (T3), it was shown in \cite{HYZ} that $B(V)$ is finite dimensional, and hence $\GKdim(B(V))=0$. 

Theorem \ref{t1.1} provides a key preliminary step for studying finite GK-dimensional pre-Nichols algebras in the twisted Yetter-Drinfeld category ${_{\k G}^{\k G} \mathcal{YD}^\Phi}$. According to this theorem, it is enough to consider pre-Nichols algebras of objects of the form $\oplus_{1\leq i\leq \theta}V_i$, where each $V_i$ is of type (T1)-(T3). Our second main theorem concerns the GKdim of pre-Nichols algebras of nondiagonal type. 

\begin{theorem}[Theorem \ref{t8.3}]\label{t1.2}
Let $V\in {_{\k G}^{\k G}\mathcal{Y}\mathcal{D}^{\Phi}}$ be a nondiagonal object. Then every pre-Nichols algebra of $V$ has infinite GKdim.
 \end{theorem}

Using this result and applying the method developed in our earlier work \cite{HLYY2,HLYY3}, we can readily classify all finite GK-dimensional Nichols algebras in ${_{\k G}^{\k G} {\mathcal{YD}^\Phi}}$. 

\begin{theorem}[Theorem \ref{t8.4}]\label{t1.3}
Let $V\in {_{\k G}^{\k G} {\mathcal{YD}^\Phi}}$ be a finite-dimensional object. Then $B(V)$ has finite GKdim if and only if it is of diagonal type and the corresponding root system is finite, that is, an arithmetic root system.
\end{theorem}

Via bosonization, this theorem also yields a classification of finite GK-dimensional coradically graded pointed coquasi-Hopf algebras over finite abelian groups that are generated by group-like and skew-primitive elements, see Corollary \ref{c8.5}.

This paper is organized as follows. In Section 2, we recall necessary definitions and facts about twisted Yetter-Drinfeld modules, pre-Nichols algebras, generalized root systems and Gelfand-Kirillov dimension.  Section 3 is devoted to graded pre-Nichols algebras and their homogenous subalgebras. In Sections 4, we classify finite GK-dimensional Nichols algebras of diagonal type in ${_{\k G}^{\k G} {\mathcal{YD}^\Phi}}$, and prove Theorem \ref{t1.1}. Sections 5-7 are devoted to the study of Nichols algebras of minimal nondiagonal objects.  Finally, in Section 8, we prove Theorems \ref{t1.2}-\ref{t1.3}.

\section{Preliminaries}
In this section, we recall some necessary notions and basic facts about twisted Yetter-Drinfeld modules, pre-Nichols algebras, generalized root systems and GK-dimension.  Throughout, let $\N$ be the set of positive integers and $\N_0=\N\cup \{0\}$, let $\Z$ be the ring of integers and $\Z_m$ a cyclic group of order $m.$

\subsection{Twisted Yetter-Drinfeld module}
Let $G$ be a finite group. It is well known that the group algebra $\k G$ is a Hopf algebra with $\D(g)=g\otimes g,\; S(g)=g^{-1}$ and
$\e(g)=1$ for any $g\in G$. Let $\Phi$ be a normalized $3$-cocycle on $G$, that is a function $\Phi: G\times G\times G\to \k^*$ such that 
\begin{eqnarray}
&\Phi(ef,g,h)\Phi(e,f,gh)=\Phi(e,f,g)\Phi(e,fg,h)\Phi(f,g,h),\\
&\Phi(f,1,g)=1
\end{eqnarray}
for all $e,f,g,h \in G$. The linear extension of $\Phi$ to $(\k G)^{\otimes 3}$, also denote by $\Phi$, becomes a
 convolution-invertible map. Let $\alpha,\beta \colon \k G \to \k$ be two linear functions defined by \[ \alpha(g):=\e(g) \quad \text{and} \quad \beta(g):=\frac{1}{\Phi(g,g^{-1},g)}, \quad \forall g\in G.\]
Then $kG$ together with these $\Phi, \ \alpha$ and $\beta$
makes a coquasi-Hopf algebra, which will be written as $(\k G,\Phi)$ in what follows. In order to distinguish the Yetter-Drinfeld modules of the group algebra $\k G$, the Yetter-Drinfeld modules of $(\k G,\Phi)$ are called twisted Yetter-Drinfeld modules when $\Phi$ is not trivial. The category of  twisted Yetter-Drinfeld modules of $(\k G,\Phi)$ will be denoted by $_{\k G}^{\k G}\mathcal{Y}\mathcal{D}^{\Phi}$.

Next we recall the construction of $_{\k G}^{\k G}\mathcal{Y}\mathcal{D}^{\Phi}$ when $G$ is a finite abelian group. For each $g\in G$, let $\Phi_g:G\times G\to \k^*$ be the function:
\begin{equation}\label{e2.3}
{\Phi}_g(x,y)=\frac{\Phi(g,x,y)\Phi(x,y,g)}{\Phi(x,g,y)},  \quad \forall x,y \in G.
\end{equation}
 By direct computation one can show that ${\Phi}_g$ is a 2-cocycle on $G$. The construction of $_{\k G}^{\k G}\mathcal{Y}\mathcal{D}^{\Phi}$ can be summarized as follows, see \cite{HY} for detailed computations.

\begin{definition}\label{d2.1}
An object in $_{\k G}^{\k G}\mathcal{Y}\mathcal{D}^{\Phi}$ is a  $G$-graded vector space $V=\oplus_{g\in G}V_g$ such that each $V_g$ is a projective
$G$-representation with respect to the 2-cocycle ${\Phi}_g,$ namely for any $e,f\in G, v\in V_g$ we have 
\begin{equation}\label{eq2.4}
e\triangleright(f\triangleright v)={\Phi}_g(e,f) (ef)\triangleright v.
\end{equation}
The comodule structure $\rho: V\to kG\otimes V$ is given by 
\begin{equation}
\rho(v)=g\otimes v, \ \  v\in V_g,\ \ g\in G.
\end{equation}
The module structure on the tensor product $V_g\otimes V_h$ is determined by
\begin{equation}\label{n3.8}
e\triangleright (X\otimes Y)={\Phi}_e(g,h)(e\triangleright X)\otimes (e\triangleright Y), \ X\in V_g, \ Y\in V_h.
\end{equation}
The associativity and the braiding constraints of $_{\k G}^{\k G}\mathcal{Y}\mathcal{D}^{\Phi}$ are given respectively by
\begin{eqnarray}
&\ a_{V_e,V_f,V_g}((X\otimes Y)\otimes Z) =\Phi(e,f,g)^{-1} X\otimes (Y\otimes Z ),\label{e2.7}\\
&\mathcal{R}(X\otimes Y)=(e \triangleright Y)\otimes X\label{eq2.7}
\end{eqnarray}
for all $X\in V_e,\  Y\in V_f,\  Z \in V_g.$
\end{definition}

Moreover, suppose $V_g$ is $(G,{\Phi}_g)$-representation, $V_h$ is a $(G,{\Phi}_h)$-representation, then $V_g\otimes V_h$ is a $(G,{\Phi}_{gh})$-representation. The dual object $V_g^*$ of $V_g$ is a $(G,{\Phi}_{g^{-1}})$-representation and $(V_g^*)^*=V_g$, see \cite[Proposition 2.5]{HYZ} for details. Suppose $V\in {_{\k G}^{\k G}\mathcal{Y}\mathcal{D}^{\Phi}}$ is a simple twisted Yetter-Drinfeld module, then there exists an element $g \in G$ such that $V=V_g$, and we say that the $G$-degree of $V$ is $g$ and denote $g_V=g$. Moreover, for each nonzero element $X\in V=V_g$, we denote $g_X=g$. Let $V=V_1\oplus V_2\oplus \cdots \oplus V_\theta$ be a direct sum of simple objects in ${^{\k G}_{\k G}\mathcal{Y}\mathcal{D}^\Phi}$, $g_i$ the $G$-degree of $V_i$ for $1\leq i\leq \theta$. Then the subgroup $G_V:=\langle g_1,g_2,\ldots, g_\theta\rangle$ will be called the {\bf support group} of $V$. According to \eqref{eq2.7}, it is clear that the braiding of $V$ is determined by the action of $G_V$.

Let $(V,\triangleright, \rho)\in {^{\k G}_{\k G}\mathcal{YD}^\Phi}$ and let $J$ be a $2$-cochain on $G$. Then there is a new action $\triangleright_J$ of $G$ on $V$ determined by
\begin{equation}\label{eq2.8}
g\triangleright_J X=\frac{J(g,x)}{J(x,g)}g\triangleright X,
\end{equation}
where $g\in G$ and $X$ is a homogenous element with $G$-degree $x=g_X$.  We denote $(V, \triangleright_{J},\rho)$ by $V^{J}$, and one can verify that $V^J\in {^{\k G}_{\k G}\mathcal{Y}\mathcal{D}^{\Phi\ast \partial(J)}}.$ Moreover there is a tensor equivalence $$(F_J,\varphi_0,\varphi_2)\colon\ ^{\k G}_{\k G}\mathcal{Y}\mathcal{D}^\Phi \to {^{\k G}_{\k G}\mathcal{Y}\mathcal{D}^{\Phi\ast\partial (J)}},$$ which takes $V$ to $V^J$ and $$\varphi_2(U,V)\colon\ (U\otimes V)^J\to U^J\otimes V^J,\ \ Y\otimes Z\mapsto J(y,z)^{-1}Y\otimes Z$$ for $Y\in U,\  Z\in V.$

\subsection{$3$-cocycles}

A further understanding of the category $_{\k G}^{\k G}\mathcal{Y}\mathcal{D}^{\Phi}$ depends on the explicit formula of the $3$-cocycle $\Phi$ on $G$.  
Recall that a complete list of the representatives of $\operatorname{H}^3(G,\k^*)$ was obtained in \cite{HLYY2}. Suppose $G= \mathbb{Z}_{m_{1}}\times\cdots \times\mathbb{Z}_{m_{n}}$ with $m_j\in \mathbb{N}$
for $1\leq j\leq n.$  Denote by $\mathcal{A}$ the set of all $\N$-sequences
\begin{equation}\label{cs}(c_{1},\ldots,c_{l},\ldots,c_{n},c_{12},\ldots,c_{ij},\ldots,c_{n-1,n},c_{123},
\ldots,c_{rst},\ldots,c_{n-2,n-1,n})\end{equation}
such that $ 0\leq c_{l}<m_{l}, \ 0\leq c_{ij}<(m_{i},m_{j}), \ 0\leq c_{rst}<(m_{r},m_{s},m_{t})$ for $1\leq l\leq n, \ 1\leq i<j\leq n, \ 1\leq r<s<t\leq n$, where $c_{ij}$ and $c_{rst}$ are ordered in the lexicographic order of their indices. We denote by $\underline{\mathbf{c}}$ the sequence \eqref{cs}  in the following. Let $g_i$ be a generator of $\mathbb{Z}_{m_{i}}, 1\leq i\leq n$. For any $\underline{\mathbf{c}}\in \mathcal{A}$, define
\begin{eqnarray}\label{eq2.10}
&& \omega_{\underline{\mathbf{c}}}\colon G\times G\times G\To \k^{\ast} \notag \\
&&[g_{1}^{i_{1}}\cdots g_{n}^{i_{n}},g_{1}^{j_{1}}\cdots g_{n}^{j_{n}},g_{1}^{k_{1}}\cdots g_{n}^{k_{n}}] \mapsto \\ 
&& \prod_{l=1}^{n}\zeta_{m_l}^{c_{l}i_{l}[\frac{j_{l}+k_{l}}{m_{l}}]}
\prod_{1\leq s<t\leq n}\zeta_{m_{t}}^{c_{st}i_{t}[\frac{j_{s}+k_{s}}{m_{s}}]}
\prod_{1\leq r<s<t\leq n}\zeta_{(m_{r},m_{s},m_{t})}^{c_{rst}i_{r}j_{s}k_{t}}. \notag
\end{eqnarray}
Here and below $\zeta_m$ stands for an $m$-th primitive root of unity.
According to \cite[Proposition 3.8]{HLYY2},   $\{\omega_{\underline{\mathbf{c}}} \mid \underline{\mathbf{c}}\in \mathcal{A}\}$ forms a complete set of representatives of the normalized $3$-cocycles on $G$ up to $3$-cohomology.

A normalized $3$-cocycle $\Phi$ on $G$ is called an abelian $3$-cocycle if $^{\k G}_{\k G}\mathcal{Y}\mathcal{D}^\Phi$ is pointed, i.e. each simple object of $^{\k G}_{\k G}\mathcal{Y}\mathcal{D}^\Phi$ is $1$-dimensional.  
Let $\Phi$ be a $3$-cocycle on $G$ of the form \eqref{eq2.10}.  In \cite[Proposition 3.14]{HLYY2}, it is proved that $\Phi$ is an abelian $3$-cocycle if and only if $c_{rst}=0$ for all $1\leq r<s<t\leq n$. So up to $3$-cohomology, an abelian $3$-cocycle $\Phi$ on $G$ must be of the form
\begin{equation}\label{eq2.11}
\Phi(g_{1}^{i_{1}}\cdots g_{n}^{i_{n}},g_{1}^{j_{1}}\cdots g_{n}^{j_{n}},g_{1}^{k_{1}}\cdots g_{n}^{k_{n}})
=\prod_{l=1}^n\zeta_l^{c_li_l[\frac{j_l+k_l}{m_l}]}\prod_{1\leq s<t\leq n}\zeta_{m_t}^{c_{st}i_t[\frac{j_s+k_s}{m_s}]}.
\end{equation}
The following lemma is obvious.
\begin{lemma}\label{l2.2}
Let $G$ be a finite cyclic group or a direct product of two finite cyclic groups. Then each normalized $3$-cocycle on $G$ is abelian.
\end{lemma}

An important observation of \cite{HLYY2} is that each abelian $3$-cocycle on $G$ can be resolved through a bigger abelian group. Suppose $G=\mathbb{Z}_{m_1}\times\cdots\times \mathbb{Z}_{m_n} = \langle g_1 \rangle \times \cdots \times \langle g_n\rangle $. Associated to $G$ there is a finite group $\widehat{G}$ defined by
\begin{equation}\label{eq2.12}
\widehat{G}=\mathbb{Z}_{m_1^2}\times\cdots\times \mathbb{Z}_{m_n^2} = \langle h_1 \rangle \times \cdots \times \langle h_n\rangle.
\end{equation}  
Let 
\begin{equation}\label{eq2.13}
\pi\colon \widehat{G}\to G,\;\;\;\;h_{i}\mapsto g_{i},\;\;\;\;\;1\le i\le n
\end{equation}
be the canonical epimorphism. The following proposition is very important. 

\begin{proposition}\cite[Proposition 3.15]{HLYY2}\label{p2.3}
Suppose that $\Phi$ is an abelian $3$-cocycle on $G.$ Then $\pi^*\Phi$ is a $3$-coboundary on $\widehat{G}$, namely, there is a $2$-cochain $J$ on $\widehat{G}$ such that $\partial J=\pi^*\Phi$.
\end{proposition}

If $\Phi$ is a nonabelian $3$-cocycle on $G,$ then it can not be resolved by any finite abelian group.
 
\begin{proposition}\cite[Proposition 3.17]{HLYY2}\label{p2.4}
Let $H$ a finite abelian group with a group epimorphism $\pi:H\to G$. If $\Phi$ is a nonabelian $3$-cocycle on $G$, then $\pi^*\Phi$ is not a $3$-coboundary on $H$.
\end{proposition} 

The braided structure of a twisted Yetter-Drinfeld module in $ {_{\k G}^{\k G}\mathcal{Y}\mathcal{D}^{\Phi}}$ is also related to $\Phi$.
Let $V\in  {_{\k G}^{\k G}\mathcal{Y}\mathcal{D}^{\Phi}}$ with braiding $\mathcal{R}$. If there is a $G$-homogenous basis $\{X_1,X_2,\ldots, X_n\}$ of $V$ and $q_{ij}\in \k^*$ such that 
\begin{equation}\label{e2.15}
\mathcal{R}(X_i\otimes X_j)=q_{ij}X_j\otimes X_i, \ \ \forall \ 1\leq i\leq j\leq n,
\end{equation}
then we say that $V$ is of diagonal type and  such basis is called a diagonal basis. Note that if $\Phi$ is an abelian $3$-cocycle on $G$, then every object in ${_{\k G}^{\k G}\mathcal{Y}\mathcal{D}^{\Phi}}$ is of diagonal type. The following lemma follows from \cite[Lemma 4.1]{HLYY2} immediately.
\begin{lemma}\label{l2.4}
Let $V\in  {_{\k G}^{\k G}\mathcal{Y}\mathcal{D}^{\Phi}}$ and let $G_V$ be the support group of $V$. Then $V$ is of diagonal type if and only if $\Phi_{G_V}$, the restriction of $\Phi$ to $G_V$, is an abelian $3$-cocycle on $G_V$.
\end{lemma}

\subsection{Nichols algebras and pre-Nichols algebras}
Nichols algebras and pre-Nichols algebras can be defined for general braided vector spaces, see \cite{HS24}. For the purpose of this paper, we only introduce these definitions for objects in the category $^{\k G}_{\k G}\mathcal{Y}\mathcal{D}^\Phi$.

Let $V$ be a finite-dimensional object in $^{\k G}_{\k G}\mathcal{Y}\mathcal{D}^\Phi.$ By $T(V)$ we denote the tensor algebra in $^{\k G}_{ \k G}\mathcal{Y}\mathcal{D}^\Phi$ generated freely by $V.$ As linear space $T(V)$ is isomorphic to $\bigoplus_{n \geq 0}V^{\otimes \overrightarrow{n}}$, where $V^{\otimes \overrightarrow{n}}$ means
$\underbrace{(\cdots((}_{n-1}V\otimes V)\otimes V)\cdots \otimes V).$ This induces a natural $\N_0$-graded structure on $T(V).$ Furthermore, $T(V)$ is an $\N_0$-graded braided Hopf algebra in $^{\k G}_{ \k G}\mathcal{Y}\mathcal{D}^\Phi$ by declaring that every nonzero element in $V$ is primitive.

A pre-Nichols algebra of $V$ is a braided Hopf algebra $\mathcal{P}(V)$ obtained as the quotient of $T(V)$ by an $\N_0$-homogenous Hopf ideal $I\subset \bigoplus_{n \geq 2}V^{\otimes \overrightarrow{n}}$. Clearly, $\mathcal{P}(V)$ is a connected $\N_0$-graded braided Hopf algebra generated by $V$. Let $I(V)$ be the maximal $\N_0$-homogenous Hopf ideal contained in $\bigoplus_{n \geq 2}V^{\otimes \overrightarrow{n}}$. The Nichols algebra $B(V)$ of $V$ is defined as the quotient Hopf algebra $T(V)/I(V)$. Thus $B(V)$ is a particular pre-Nichols algebra, and there exist braided Hopf algebra epimorphisms
$$T(V)\twoheadrightarrow \mathcal{P}(V)\twoheadrightarrow B(V),$$
Where the restrictions of both epimorphisms to degree one are $\id_V$.

Let $V=V_1\oplus \cdots \oplus V_\theta\in {^{\k G}_{\k G}\mathcal{YD}^\Phi}$ be a direct sum of $\theta$ simple objects. Any pre-Nichols algebra $\mathcal{P}(V)$ of $V$ is said to be of rank $\theta$, and we write $\Rank(\mathcal{P}(V))=\theta$.
Let $g_i=g_{V_i}$ for $1\leq i\leq \theta$. If the set $\{g_1,\ldots, g_\theta\}$ freely generates $G$, i.e., $G\cong \langle g_1\rangle \times \cdots \times \langle g_\theta\rangle$, then $V$ and the associated pre-Nichols algebras are called of {\bf standard form}.

For each braided Hopf algebra $H$ in $^{\k G}_{\k G}\mathcal{Y}\mathcal{D}^\Phi$ and a $2$-cochain $J$ on $G$, one can define a new braided Hopf algebra  $H^J$ in $^{\k G}_{\k G}\mathcal{Y}\mathcal{D}^{\Phi\ast \partial J}$ whose multiplication $\circ$ is given by
\begin{equation}\label{eq2.14}
X\circ Y=J(x,y)XY
\end{equation} for all homogenous elements $X,Y\in H,$ where $x=g_X$ and $y=g_Y$ denote $G$-degrees of $X$ and $Y$ respectively.  In particular, let $(V,\triangleright, \rho)\in {^{\k G}_{\k G}\mathcal{YD}^\Phi}$ and let $\mathcal{P}(V)$ be a pre-Nichols algebra of $V$. Then 
  $\mathcal{P}(V)^J$ is also an $\N_0$-graded braided Hopf algebra in $^{\k G}_{\k G}\mathcal{Y}\mathcal{D}^{\Phi\ast \partial J}$, generated by $V^J=(V,\triangleright_J, \rho)$.  This leads to the following lemma.
 
 \begin{lemma}\label{l2.5}
Let $\mathcal{P}(V)$ be a pre-Nichols algebra of $V\in {^{\k G}_{\k G}\mathcal{YD}^\Phi}$ and let $J$ be a $2$-cochain on $G$. Then $\mathcal{P}(V)^J$ is a pre-Nichols algebra of $V^J\in {^{\k G}_{\k G}\mathcal{Y}\mathcal{D}^{\Phi\ast \partial J}}$.
\end{lemma} 

Adopting the same terminology as in the theory of coquasi-Hopf algebras, we say that $\mathcal{P}(V)$ and $\mathcal{P}(V)^{J}$ are twist equivalent, or $\mathcal{P}(V)^J$ is a twisting of $\mathcal{P}(V)$. A special case of Lemma \ref{l2.5} was given in \cite{HLYY2}. 
\begin{lemma}\cite[Lemma 2.12]{HLYY2}\label{l2.6}
The twisting $B(V)^J$ of $B(V)$ is a Nichols algebra in $^{\k G}_{\k G}\mathcal{Y}\mathcal{D}^{\Phi\ast \partial J}$ and $B(V)^J\cong B(V^J)$.
\end{lemma}

\subsection{Dynkin diagram and root system}
Arithmetic root systems are invariants for Nichols algebras of diagonal type with certain finiteness property. A complete classification of arithmetic root systems, established by Heckenberger \cite{H4}, plays a crucial role in the classification program of pointed Hopf algebras with finitely many PBW generators.

Suppose $B(V)$ is a Nichols algebra of diagonal type in $^{\k G}_{\k G}\mathcal{Y}\mathcal{D}^{\Phi}.$ Fix a diagonal basis $\{X_i|1\leq i\leq n\}$ of $V$. As shown in \cite{HLYY2}, $B(V)$ has a $\Z^n$-graded braided Hopf algebra structure with $\deg(X_i)=e_i, 1\le i\leq n$, where $\{e_i|1\leq i\leq n\}$ is a canonical basis of $\Z^n.$ Moreover, $B(V)$ admits a PBW basis whose generators are homogenous with respect to this $\Z^n$-grading. Following \cite{H0}, we denote by $\triangle^{+}(B(V))$ the set of $\Z^n$-degrees of these PBW generators, counted with multiplicities, and define the root system of $B(V)$ as $$\bigtriangleup(B(V)):= \bigtriangleup^{+}(B(V))\bigcup -\bigtriangleup^{+}(B(V)).$$ When $\bigtriangleup(B(V))$ is finite,  it coincides with an arithmetic root system.  

To each Nichols algebra of diagonal type one can also associated a diagram. 
Let $V$ be a diagonal object in ${^{\k G}_{\k G}\mathcal{Y}\mathcal{D}^{\Phi}}$ with a diagonal basis $\{X_i|1\leq i\leq n\}$ and braiding $\mathcal{R}$.
The structure constants of $B(V)$ are $\{q_{ij}|1\leq i,j\leq n\}$ satisfying \eqref{e2.15}. The generalized Dynkin diagram $\mathcal{D}(V)$ is defined as an undirected graph as follows:
\begin{itemize}
\item[1)] There is a bijecton $\phi$ from the index set $I=\{ 1, 2, \dots, n \}$ to the set of vertices of $\mathcal{D}(V)$.
\item[2)]  The vertex $\phi(i)$ is labelled  by $q_{ii}$ for all $1\leq i\leq n$.
\item[3)] For $1\leq i,j\leq n,$ the number  $n_{ij}$ of edges between $\phi(i)$ and $\phi(j)$ is either $0$ or $1.$ If $i=j$ or $q_{ij}q_{ji}=1$, then $n_{ij}=0$. Otherwise $n_{ij}=1$ and the unique edge is labelled by $\widetilde{q_{ij}}=q_{ij}q_{ji}$ for all 
$1\leq i<j\leq n.$
\end{itemize}

A Nichols algebra of diagonal type always admits a generalized Dynkin diagram, and this diagram is independent of the choice of diagonal basis for $V$.

\begin{definition}
Let $V\in  {_{\k G}^{\k G}\mathcal{YD}^\Phi}$ be a twisted Yetter-Drinfeld module of diagonal type. The generalized Dynkin diagram $\mathcal{D}(V)$ is said to be of finite type if $\bigtriangleup(B(V))$ is finite.
\end{definition}

Let $J$ be a $2$-cochain on $G$ and let $V^J$ be the twisting of $V$ by $J$. Denote by $\mathcal{R}_J$ the braiding of $V^J$ and by $\{\overline{q_{ij}}|1\leq i,j\leq n\}$ be the corresponding structure constants. Using \eqref{eq2.8} we obtain
$$\mathcal{R}_J(X_i\otimes X_j)=g_i\triangleright_J X_j\otimes X_i=\frac{J(g_i, g_j)}{J(g_j,g_i)}g_i\triangleright X_j\otimes X_i=\frac{J(g_i, g_j)}{J(g_j,g_i)}q_{ij}X_j\otimes X_i,$$
where $g_i$ is the $G$-degree of $X_i$ for $1\leq i\leq n$. Hence $\overline{q_{ij}}=\frac{J(g_i, g_j)}{J(g_j,g_i)}q_{ij}$, which yields $$\overline{q_{ii}}=q_{ii},\ \  \overline{q_{ij}}\overline{q_{ji}}=q_{ij}q_{ji},\ \ \forall \ 1\leq i,j \leq n.$$ 
Consequently, $\mathcal{D}(V)$ and $\mathcal{D}(V^J)$ are the same generalized Dynkin diagram. Moreover, by \eqref{eq2.14} the twisting preserves the $Z^n$-degree of every $Z^n$-homogenous element of $B(V)$. Therefore we obtain the following proposition. 

\begin{proposition}\label{p2.10}
Let $B(V)\in {_{\k G}^{\k G}\mathcal{YD}^\Phi}$ be a Nichols algebra of diagonal type and let $J$ be a $2$-cochain on $G$. Then $B(V)$ and its twisting $B(V^J)$ share the same root system and the same generalized Dynkin diagram.
\end{proposition}

\subsection{On GK-dimension}
Let $A$ be a finitely generated algebra, and let $V$ be a finite-dimensional generating subspace of $A$. Set $A_n=\sum_{0\leq i\leq n}V^i$. The Gelfand-Kirillov dimension (GK-dimension or GKdim for short) of $A$ is defined by 
\begin{equation}
\GKdim(A)=\overline{\lim}_{n\to \infty}\log_n {\dim(A_n)}.
\end{equation}
The value of GK-dimension does not depend on the choice of finite-dimensional generating subspace. If $A$ is not finitely generated, its GK-dimension is defined as
\begin{equation}
\GKdim(A)=\sup\{\GKdim(B) \ | \ B\ \mathrm{is} \ \mathrm{a}\ \mathrm{finitely}\ \mathrm{generated}\ \mathrm{subalgebra} \ \mathrm{of}\  A\}.
\end{equation}
Let $\mathcal{P}(V)$ be a pre-Nichols algebra of $V\in {_{\k G}^{\k G} \mathcal{YD}^\Phi}$. For each $n\in \N$, let $$\mathcal{P}(V)_n=\sum_{0\leq i\leq n}V^i\subset \mathcal{P}(V),\ \ \ \ B(V)_n=\sum_{0\leq i\leq n}V^i\subset B(V).$$
Since the canonical epimorphism from $\mathcal{P}(V)$ onto $B(V)$ is $\N_0$-graded, we have $\dim(\mathcal{P}(V)_n)\geq \dim(B(V)_n)$ for all $n\in \N$ and consequently 
\begin{equation}\label{eq2.18}
\GKdim(\mathcal{P}(V))\geq \GKdim(B(V)).
\end{equation}
Furthermore, if $B\subset \mathcal{P}(V)$ is a finitely generated $\N_0$-graded subalgebra, then it is clear that 
\begin{equation}\label{eq2.19}
\GKdim(B)\leq \GKdim(\mathcal{P}(V)).
\end{equation}

Let $J$ be a $2$-cochain on $G$. From \eqref{eq2.14} we obtain $(\mathcal{P}(V)^J)_n\subset \mathcal{P}(V)_n$ as linear spaces for each $n\in \N$. This implies that 
\begin{equation}\label{eq2.20}
\GKdim(\mathcal{P}(V)^J)\leq \GKdim(\mathcal{P}(V)).
\end{equation}
On the other hand, Applying \eqref{eq2.20} again gives
\begin{equation}\label{eq2.21}
\GKdim(\mathcal{P}(V))=\GKdim(((\mathcal{P}(V)^J)^{J^{-1}})\leq \GKdim(\mathcal{P}(V)^J).
\end{equation}
Thus we obtain the following lemma.

\begin{lemma}\label{l2.11}
Let $\mathcal{P}(V)$ be a pre-Nichols algebra of $V\in {_{\k G}^{\k G} \mathcal{YD}^\Phi}$, and let $J$ be a $2$-cochain on $G$. Then $$\GKdim(\mathcal{P}(V)^J)= \GKdim(\mathcal{P}(V)).$$
\end{lemma}

\section{Graded pre-Nichols algebras}
In this section, we explore the structural properties of graded  pre-Nichols algebras in ${_{\k G}^{\k G} \mathcal{YD}^\Phi}$, with a particular emphasis on the associated based groups and homogenous subalgebras.  
Let $V=V_1\oplus \cdots \oplus V_\theta$ be a direct sum of simple objects in the category ${_{\k G}^{\k G} \mathcal{YD}^\Phi}$. Recall that a pre-Nichols algebra $\mathcal{P}(V)$ is said to be {\bf graded} if it is an $\N_0^\theta$-graded braided Hopf algebra such that $\deg(V_i)=e_i$ for $1\leq i\leq \theta$, where $\{e_1,\ldots ,e_\theta\}$ denotes the standard basis of $\N_0^\theta$. 

 \subsection{Based groups}

 Let $H$ be a braided Hopf algebra in $_{\k G}^{\k G} \mathcal{YD}^\Phi$. We call $G$ the based group of $H$, and denote by $G_H$ be the support group of $H$.  For any $V\in {_{\k G}^{\k G} \mathcal{YD}^\Phi}$ and a pre-Nichols algebra $\mathcal{P}(V)$, it is clear that $G_{\mathcal{P}(V)}=G_V$. A useful technique in the study of Nichols algebras, introduced in \cite{HLYY2,HLYY3}, is to realize a given Nichols algebra in different twisted Yetter-Drinfeld module categories. This method can be extended naturally to graded pre-Nichols algebras.

Let $H$ and $M$ be braided Hopf algebras in ${_{\k G}^{\k G} \mathcal{YD}^\Phi}$ and ${_{\k F}^{\k F} \mathcal{YD}^\Psi}$ respectively, where $F$ is also a finite abelian group and $\Psi$ is a $3$-cocycle on $F$. We say that $H$ and $M$ are isomorphic if there exists a linear isomorphism $f:H\to M$ that is both an algebra homomorphism and a coalgebra homomorphism.

\begin{lemma}\label{l3.1}
Let $H$ be a braided Hopf algebra in ${_{\k G}^{\k G} \mathcal{YD}^\Phi}$ and let $F$ be a subgroup of $G$ such that $G_H\subset F$. Then $H$ is also a braided Hopf algebra in ${_{\k F}^{\k F} \mathcal{YD}^{\Psi}}$, where $\Psi=\Phi_{F}$ is the restriction of $\Phi$ to $F$. In particular, $H$ is a braided Hopf algebra in ${_{\k G_H}^{\k G_H} \mathcal{YD}^{\Phi_{G_H}}}$.
\end{lemma}
\begin{proof}
Let $\rho: H\to \k G\otimes H$ be the comodule structure map. Since $\rho(H)\subset \k G_H\otimes H\subset \k F\otimes H$, the space $H$ inherits an $F$-comodule structure, and thus $H=\oplus_{g\in F} H_g$.  Take $X\in H_g$ and $e,f\in F$. Using  \eqref{eq2.4} we obtain $$e\triangleright(f\triangleright X)={\Phi}_g(e,f)ef\triangleright X={\Psi}_g(e,f) ef\triangleright X.$$ Hence each $H_g$ is an $(F,\Psi_g)$-projective representation, and consequently $H$ is a twisted Yetter-Drinfeld module in the category ${_{\k F}^{\k F} \mathcal{YD}^{\Psi}}$.
 
 It remains to show that $H$ is a braided Hopf algebra in ${_{\k F}^{\k F} \mathcal{YD}^{\Psi}}$. Let $m$ and $\bigtriangleup$ denote the product and coproduct of $H$, respectively. Since $m$ and $\bigtriangleup$ are morphisms in ${_{\k G}^{\k G} \mathcal{YD}^\Phi}$, for all $g\in F\subset G$ and $X,Y\in H$ we have 
\begin{eqnarray}
&g\triangleright m(X\otimes Y)=m(g\triangleright(X\otimes Y)),\label{eq3.1}\\
& X_1Y_1\otimes m(X_2\otimes Y_2)=m(XY)_1\otimes m(XY)_2,\label{eq3.2}
\end{eqnarray}
 where we use the Sweedler notation $\rho(X)=X_1\otimes X_2$ for all $X\in H$. Because $X_1Y_1, m(XY)_1\in G_H\subset F$, the identities \eqref{eq3.1}-\eqref{eq3.2} imply that $m$ is a morphism in ${_{\k F}^{\k F} \mathcal{YD}^{\Psi}}$. A similar argument shows that $\bigtriangleup$ is also a morphism in the same category. Thus $H$ is indeed a braided Hopf algebra in ${_{\k F}^{\k F} \mathcal{YD}^{\Psi}}$.
\end{proof}

We now present an alternative method for changing the based group of a graded pre-Nichols algebra. Let $\mathbbm{G}$ be a finite abelian group and let $\pi: \mathbbm{G}\to G$ be a group epimorphism. Define the $3$-cocycle $\pi^*{\Phi}$ on $\mathbbm{G}$  by $$\pi^*{\Phi}(e,f,g)=\Phi(\pi(e),\pi(f),\pi(g)),\ \forall e,f,g\in \mathbbm{G}.$$
If $\mathcal{P}(V)\in {_{\k G}^{\k G} \mathcal{YD}^\Phi}$ is graded, we shall show that $\mathcal{P}(V)$ can also be realized in ${_{\k \mathbbm{G}}^{\k \mathbbm{G}}\mathcal{YD}^{\pi^*{\Phi}}}$.

\begin{lemma}\label{l3.3}
Let $\mathbbm{G}$ be a finite abelian group and let $\pi: \mathbbm{G}\to G$ be a group epimorphism that admits a section $\iota: G\to \mathbbm{G}$, i.e., $\pi\circ \iota=\id_G$. For an object $(V,\triangleright, \rho)$ in ${_{\k G}^{\k G} \mathcal{YD}^\Phi}$, define 
\begin{eqnarray}
&\widetilde{ \rho}=(\iota\otimes \id)\rho,\label{eq3.3}\\
&\mathbbm{g}\ \widetilde{ \triangleright }X=\pi(\mathbbm{g})\triangleright X . \label{eq3.4}
\end{eqnarray}
Then $\widetilde{V}=(V, \widetilde{ \triangleright }, \widetilde{ \rho})$ is a twisted Yetter-Drinfeld module over $(\k \mathbbm{G}, \pi^*\Phi)$. Moreover, if $V$ is simple, so is $\widetilde{V}$.
\end{lemma}
\begin{proof}  
Since $V$ is a $G$-graded, we may write $V=\bigoplus_{g\in G}V_g$. Then $V$ inherits an $G$-grading by setting 
\[ V_\mathbbm{g}=
\begin{cases}
V_g & \ \mathrm{if}\ \mathbbm{g}=\iota(g)\ \mathrm{for}\ \mathrm{some}\ g\in G, \\
\ 0, & \ \mathrm{otherwise}. 
\end{cases}\]
To show that $\widetilde{V}$ belongs to $(\k \mathbbm{G}, \pi^*\Phi)$, it suffices to verify that each homogenous component ${V}_{\iota(g)}$ is a projective $\mathbbm{G}$-representation with respect to 
 2-cocycle ${\pi^*{\Phi}}_{\iota(g)}$. Take $e,f\in \mathbbm{G}$ and $v\in {V}_{\iota(g)}$, we have 
\begin{eqnarray*}
&&e\ \widetilde{ \triangleright }(f\ \widetilde{ \triangleright } v)= \pi(e)\triangleright (\pi(f)\triangleright v)\\
&=&{\Phi}_g(\pi(e),\pi(f)) \pi(ef)\triangleright v\\
&=&\frac{\Phi(g,\pi(e),\pi(f))\Phi(\pi(e),\pi(f),g)}{\Phi(\pi(e),g,\pi(f))} ef\ \widetilde{ \triangleright } v\\
&=&\frac{\Phi(\pi\circ\iota(g),\pi(e),\pi(f))\Phi(\pi(e),\pi(f),\pi\circ\iota(g))}{\Phi(\pi(e),\pi\circ\iota(g),\pi(f))} ef\ \widetilde{ \triangleright }v\\
&=&\frac{\pi^*\Phi(\iota(g),e,f)\pi^*\Phi(e,f,\iota(g))}{\pi^*\Phi(e,\iota(g),f)} ef\ \widetilde{ \triangleright }v\\
&=&{\pi^*\Phi}_{\iota(g)}(e,f) ef\ \widetilde{ \triangleright } v.
\end{eqnarray*}

Thus ${V}_{\iota(g)}$ is indeed a projective $\mathbbm{G}$-representation with respect to ${\pi^*{\Phi}}_{\iota(g)}$, and consequently $\widetilde{V}\in {_{\k \mathbbm{G}}^{\k \mathbbm{G}}\mathcal{YD}^{\pi^*{\Phi}}}$.

Now suppose $V$ is simple. If $\widetilde{V}$ is not simple, it would contain a nontrivial subobject $\widetilde{U}$. Because $\widetilde{V}=\bigoplus_{\mathbbm{g}\in \mathbbm{G}}V_\mathbbm{g}=\bigoplus_{g\in G}V_{\iota(g)}$, we have $\widetilde{U}=\bigoplus_{g\in G}U_{\iota(g)}$ with $U_{\iota(g)}\subset V_{\iota(g)}$. Define $U=\bigoplus_{g\in G}U_{g}$, where $U_{g}=U_{\iota(g)}$ as vector spaces. Then $U$ becomes an object in ${_{\k G}^{\k G} \mathcal{YD}^\Phi}$ by defining $g\triangleright u=\iota(g)\widetilde{\triangleright}\ u$ for all $g\in G$ and $u\in U$. It is clear that $U$ is a nontrivial subobject of $V$, which is a contradiction since $V$ is simple. Thus $\widetilde{V}$ is also simple.
\end{proof}

\begin{lemma}\label{l3.4}
 Let $\mathbbm{G}$ be a finite abelian group and let $\pi: \mathbbm{G}\to G$ be  a group epimorphism that admits a section $\iota:G\to \mathbbm{G}$. Then every graded pre-Nichols algebra $\mathcal{P}(V)$ in  ${_{\k G}^{\k G} \mathcal{YD}^\Phi}$ is isomorphic to a graded pre-Nichols algebra in the category ${_{\k \mathbbm{G}}^{\k \mathbbm{G}} \mathcal{YD}^{\pi^*\Phi}}$.
 \end{lemma}
\begin{proof}
Let $\widetilde{V}$ be the twisted Yetter-Drinfeld module in ${_{\k \mathbbm{G}}^{\k \mathbbm{G}}\mathcal{YD}^{\pi^*{\Phi}}}$ obtained from $V$ via \eqref{eq3.3}-\eqref{eq3.4}. Denote by $T(V)$ and $T(\widetilde{V})$ the tensor algebras of $V\in {_{\k G}^{\k G} \mathcal{YD}^\Phi}$ and of $\widetilde{V}\in {_{\k \mathbbm{G}}^{\k \mathbbm{G}}\mathcal{YD}^{\pi^*{\Phi}}}$, respectively. As vector spaces, $\widetilde{V}=V$, hence $T(V)=T(\widetilde{V})$ and the identity map $\epsilon: T(V)\to T(\widetilde{V})$ is an algebra isomorphism, because both tensor algebras are freely generated by the same underlying vector space. 

We now show that $\epsilon$ is also a coalgebra morphism, and therefore an isomorphism of braided Hopf algebras. Write $V=V_1\oplus \cdots \oplus V_\theta$ as a direct sum of simple objects. By Lemma \ref{l3.3},   
$\widetilde{V}=\widetilde{V}_1\oplus \cdots \oplus\widetilde{V}_\theta$ is also a direct sum of simple objects. Clearly, $T(V)$ and $T(\widetilde{V})$ are both $\N_0^\theta$-graded with $\deg(V_i)=\deg(\widetilde{V}_i)=e_i$ for $1\leq i\leq \theta$. For a homogenous element $X\in T(V)$ of degree $\deg(X)=\sum_{i=1}^\theta k_ie_i$, denote by $g_X$ its $G$-degree in $T(V)$ and by $\mathbbm{g}_X$ its $\mathbbm{G}$-degree in $T(\widetilde{V})$. Then 
\begin{equation}\label{eq3.5}
g_X=g_1^{k_1}g_2^{k_2}\cdots g_\theta^{k_\theta},\ \ \mathbbm{g}_X=\iota(g_1)^{k_1}\iota(g_2)^{k_2}\cdots \iota(g_\theta)^{k_\theta}
\end{equation}
and $\pi(\mathbbm{g}_X)=g_X.$

Let $\D$ and $\widetilde{\D}$ be the coproducts of $T(V)$ and $T(\widetilde{V})$ respectively. We will prove that 
\begin{equation}
(\epsilon\otimes \epsilon)\D=\widetilde{\D}\circ \epsilon
\end{equation}
 by induction on the lengths of $\N_0^\theta$-homogenous elements. For $X\in V$, both $\D(X)$ and $\widetilde{\D}(X)$ equal to $X\otimes 1 +1\otimes X$, thus the claim holds for generators. 
 
Assume the statement is true for all elements of length $n\geq 1$. Take a nomogenous element $X\in V^{n+1}$. We may write $X=\sum_{i\in I} X_iY_i$ with $X_i\in V$ and $Y_i\in V^n$ homogenous. Using Sweedler's notation $\D(Y_i)=Y_{i1}+Y_{i2}$ and the induction hypothesis, we have  
\begin{equation}\label{eq1}
\begin{split}
&(\epsilon\otimes \epsilon)\D(X_iY_i)\\
&=(\epsilon\otimes \epsilon)(X_i\otimes 1+1\otimes X_i)(Y_{i1}\otimes Y_{i2})\\
&=(\epsilon\otimes \epsilon)[\Phi(g_{X_i}, g_{Y_{i1}},g_{Y_{i2}}) X_iY_{i1}\otimes Y_{i2}+\frac{\Phi(g_{X_i}, g_{Y_{i1}},g_{Y_{i2}})}{\Phi(g_{Y_{i1}},g_{X_i},g_{Y_{i2}})} g_{X_i}\triangleright Y_{i1}\otimes X_iY_{i2}]\\
&=\Phi(g_{X_i}, g_{Y_{i1}},g_{Y_{i2}}) X_iY_{i1}\otimes Y_{i2}+\frac{\Phi(g_{X_i}, g_{Y_{i1}},g_{Y_{i2}})}{\Phi(g_{Y_{i1}},g_{X_i},g_{Y_{i2}})} g_{X_i}\triangleright Y_{i1}\otimes X_iY_{i2},
\end{split}
\end{equation}
and 
\begin{equation}\label{eq2}
\begin{split}
&\widetilde{\D}\circ \epsilon(X_iY_i)=\widetilde{\D}(\epsilon(X_i)\epsilon(Y_i))\\
&=(X_i\otimes 1+1\otimes X_i)(Y_{i1}\otimes Y_{i2})\\
&=\pi^*\Phi(\mathbbm{g}_{X_i}, \mathbbm{g}_{Y_{i1}},\mathbbm{g}_{Y_{i2}})X_iY_{i1}\otimes Y_{i2}\\
&\ \ \ \ \  \ \ \ \ \ \ \ \ \ \ \ + \frac{\pi^*\Phi(\mathbbm{g}_{X_i}, \mathbbm{g}_{Y_{i1}},\mathbbm{g}_{Y_{i2}})}{\pi^*\Phi(\mathbbm{g}_{Y_{i1}},\mathbbm{g}_{X_i},\mathbbm{g}_{Y_{i2}})}\mathbbm{g}_{X_i}\widetilde{\triangleright}\ Y_{i1} \otimes X_iY_{i2}\\
&=\Phi(\pi(\mathbbm{g}_{X_i}), \pi(\mathbbm{g}_{Y_{i1}}),\pi(\mathbbm{g}_{Y_{i2}}))X_iY_{i1}\otimes Y_{i2}\\
&\ \ \ \ \ \ \ \ \ \ \ \ \ \ \ \ +\frac{\Phi(\pi(\mathbbm{g}_{X_i}), \pi(\mathbbm{g}_{Y_{i1}}),\pi(\mathbbm{g}_{Y_{i2}}))}{\Phi(\pi(\mathbbm{g}_{Y_{i1}}),\pi(\mathbbm{g}_{X_i}),\pi(\mathbbm{g}_{Y_{i2}}))}\mathbbm{g}_{X_i}\widetilde{\triangleright}\ Y_{i1} \otimes X_iY_{i2}\\
&=\Phi(g_{X_i}, g_{Y_{i1}},g_{Y_{i2}}) X_iY_{i1}\otimes Y_{i2}+\frac{\Phi(g_{X_i}, g_{Y_{i1}},g_{Y_{i2}})}{\Phi(g_{Y_{i1}},g_{X_i},g_{Y_{i2}})} g_{X_i}\triangleright Y_{i1}\otimes X_iY_{i2}.
\end{split}
\end{equation}

Thus  $(\epsilon\otimes \epsilon)\D(X_iY_i)=\widetilde{\D}\circ \epsilon(X_iY_i)$ for each $i\in I$, whence $(\epsilon\otimes \epsilon)\D(X)=\widetilde{\D}\circ \epsilon (X)$.
This shows that $\epsilon$ is a coalgebra morphism, and consequently an isomorphism of braided Hopf algebras between $T(V)$ and $T(\widetilde{V})$.

Finally, we show that $\mathcal{P}(V)$ is isomorphic to a quotient of $T(\widetilde{V})$. Let $I$ be the braided Hopf ideal of $T(V)$ such that $\mathcal{P}(V)=T(V)/I$. Because $\mathcal{P}(V)$ is graded, $I$ must be $\N_0^\theta$-graded. Set $\widetilde{I}=\epsilon(I)$. It is a braided Hopf ideal of $T(\widetilde{V})$, hence $T(\widetilde{V})/\widetilde{I}$ is a braided Hopf algebra in ${_{\k \mathbbm{G}}^{\k \mathbbm{G}} \mathcal{YD}^{\pi^*(\Phi)}}$. Clearly, 
the map $\epsilon$ induces an isomorphism from $\mathcal{P}(V)=T(V)/I$ to $T(\widetilde{V})/\widetilde{I}\in {_{\k \mathbbm{G}}^{\k \mathbbm{G}} \mathcal{YD}^{\pi^*(\Phi)}}$, which proves the lemma. 
\end{proof}

Recall that for each finite abelian group $G$, there exists an abelian group $\widehat{G}$ and a canonical group epimorphism $\pi:\widehat{G}\to G$ defined by \eqref{eq2.13}. With these notations we have

\begin{proposition}\label{p3.4}
Let $\mathcal{P}(V)$ be a graded pre-Nichols algebra of diagonal type in the category ${_{\k G}^{\k G} \mathcal{YD}^\Phi}$. Then $\mathcal{P}(V)$ is twist equivalent to a graded pre-Nichols algebra of diagonal type in the ordinary Yetter-Drinfeld module category ${_{\k \widehat{G_V}}^{\k \widehat{G_V}} \mathcal{YD}}$.
\end{proposition}
\begin{proof}
By Lemma  \ref{l3.1}, $\mathcal{P}(V)$ is already a pre-Nichols algebra in ${_{\k G_V}^{\k G_V} \mathcal{YD}^{\Phi_{G_V}}}$. Applying Lemma \ref{l3.4}, we see that $\mathcal{P}(V)$ is isomorphic to a graded pre-Nichols algebra in ${_{\k \widehat{G_V}}^{\k \widehat{G_V}} \mathcal{YD}^{\pi^*\Phi_{G_V}}}$. Now, because $\mathcal{P}(V)$ is of diagonal type, Lemma \ref{l2.4}  implies that $\Phi_{G_V}$ is an abelian $3$-cocycle on $G_V$. Hence,  by Proposition \ref{p2.3}, the pulled-back cocycle  $\pi^*\Phi_{G_V}$ is a $3$-coboundary on $\widehat{G_V}$. Consequently, there exists a $2$-cochain $J$ on $\widehat{G_V}$ such that $\partial J=\pi^*\Phi_{G_V}$. Finally, Lemma \ref{l2.5} shows that the twisting $\mathcal{P}(V)^{J^{-1}}$ is a graded pre-Nichols algebra in the ordinary Yetter-Drinfeld category ${_{\k \widehat{G_V}}^{\k \widehat{G_V}} \mathcal{YD}}$.
\end{proof}

Since Nichols algebras are naturally graded, we obtain the following corollary immediately by Lemma \ref{l2.6} and Proposition \ref{p3.4}.
\begin{corollary}\label{c3.5}
Let $B(V)$ be a Nichols algebra of diagonal type in ${_{\k G}^{\k G} \mathcal{YD}^\Phi}$. Then $B(V)$ is twist equivalent to a Nichols algebra of diagonal type in the ordinary Yetter-Drinfeld category 
${_{\k \widehat{G_V}}^{\k \widehat{G_V}} \mathcal{YD}}$.
\end{corollary}

 \subsection{Homogenous subalgebras}
 
 Let $\mathcal{P}(V)$ be a graded pre-Nichols algebra in ${_{\k G}^{\k G} \mathcal{YD}^\Phi}$. For any homogenous subobject  $I$ of $\mathcal{P}(V)$, we denote by $\Deg(I)$ the set of $\N_0^\theta$-degrees of nonzero homogenous elements of $I$. The main result of this subsection is as follows.
 \begin{proposition}\label{p3.6}
 Let $\mathcal{P}(V)$ be a graded pre-Nichols algebra in ${_{\k G}^{\k G} \mathcal{YD}^\Phi}$ with counit $\epsilon$, and let $A$ be an homogenous subalgebra of $\mathcal{P}(V)$. Assume that there exist homogenous subobjects $M,N\subset \ker \epsilon$ satisfying the following conditions:
 \begin{itemize}
\item[(a)] $\Deg(A) \cap \Deg(M)=\emptyset$ and $\Deg(A) \cap \Deg(N)=\emptyset$;
\item[(b)] $M$ and $N$ are $A$-bimodules under the product of $\mathcal{P}(V)$, and $M^2\subset M, N^2\subset N$;
\item[(c)] $\D(A)\subset A\otimes A+M\otimes N$.
 \end{itemize}
Then $A$ is a braided bialgebra with a new comultiplication $\widetilde{\D}=P\circ \D$, where $$P:A\otimes A+ M\otimes N \to A\otimes A$$ is the canonical projection.
 \end{proposition}
 \begin{proof}
 First, we show that $\widetilde{\D}$ is coassociative. By condition (c), we have 
 \begin{eqnarray}
 (\D\otimes \id)\D(A)&\subset & (A\otimes A)\otimes A+(M\otimes N)\otimes A+ (P(V)\otimes P(V))\otimes N. \label{equa4.1}\\
 (\id\otimes \D)\D(A)&\subset&  A\otimes (A\otimes A)+A\otimes (M\otimes N)+ M\otimes (P(V)\otimes P(V)).\label{equa4.2}
 \end{eqnarray}
 By condition (a), the right hand sides of \eqref{equa4.1} and \eqref{equa4.2} are direct sums. For a homogenous element $x\in A$,  we can write  
\begin{equation}\label{eq3.11}
(\D\otimes \id)\D(x)=(\widetilde{\D}\otimes \id)\widetilde{\D}(x)+ (x'_1\otimes x'_2)\otimes x'_3+ (x''_1\otimes x''_2)\otimes x''_3
\end{equation}
with $(x'_1\otimes x'_2)\otimes x'_3\in (M\otimes N)\otimes A$ and $(x''_1\otimes x''_2)\otimes x''_3\in (P(V)\otimes P(V))\otimes N$.
Similarly, 
\begin{equation}\label{eq3.12}
(\id\otimes \D)\D(x)=(\id\otimes \widetilde{\D} )\widetilde{\D}(x)+ Y'_1\otimes (Y'_2\otimes Y'_3)+ Y''_1\otimes (Y''_2\otimes Y''_3),
\end{equation}
where $Y'_1\otimes (Y'_2\otimes Y'_3)\in A\otimes (M\otimes N)$ and $Y''_1\otimes (Y''_2\otimes Y''_3)\in M\otimes (P(V)\otimes P(V))$.
Let $a$ denote the associator of the category ${_{\k G}^{\k G} \mathcal{YD}^\Phi}$.  From the coassociativity of $\D$ we have $$a\circ (\D\otimes \id)\D(x)=(\id\otimes \D)\D(x).$$ 
Substituting the two decompositions \eqref{eq3.11} and \eqref{eq3.12}, and considering the $\N_0^\theta$-degrees of the three summands in each decomposition, we obtain 
$$a\circ (\widetilde{\D}\otimes \id)\widetilde{\D}(x)=(\id\otimes \widetilde{\D} )\widetilde{\D}(x).$$ Hence $\widetilde{\D}$ is coassociative.

Next we verify that $\widetilde{\D}$ is an algebra homomorphism. For $X, Y\in A$, we write 
\begin{equation*}
\begin{split}
 \D(X)&=\widetilde{\D}(X)+X'_1\otimes X'_2, \\
\D(Y)&=\widetilde{\D}(Y)+Y'_1\otimes Y'_2
 \end{split}
 \end{equation*}
  with $X'_1\otimes X'_2, Y'_1\otimes Y'_2 \in M\otimes N$. Then we have 
\begin{equation*}
\begin{split}
\widetilde{\D}(XY)&=P\circ \D(XY)\\
&=P\big((\widetilde{\D}(X)+X'_1\otimes X'_2)(\widetilde{\D}(Y)+Y'_1\otimes Y'_2)\big)\\
&=P\big(\widetilde{\D}(X)\widetilde{\D}(Y)+\widetilde{\D}(X)(Y'_1\otimes Y'_2)+(X'_1\otimes X'_2)\widetilde{\D}(Y)+(X'_1\otimes X'_2)(Y'_1\otimes Y'_2)\big)\\
&=\widetilde{\D}(X)(\widetilde{\D}(Y).
\end{split}
\end{equation*}
Here the fourth identity follows from the condition (b). Thus $\widetilde{\D}$ is an algebra homomorphism.

Finally, we check the counit condition, that is $(\epsilon \otimes \id)\widetilde{\D}=(\id \otimes \epsilon)\widetilde{\D}=\id$. For each $X\in A$, we write $\D(X)=\widetilde{\D}(X)+X'_1\otimes X'_2$ with $X'_1\otimes X'_2\in M\otimes N$. So $(\epsilon \otimes \id)(X'_1\otimes X'_2)=0$ since $M,N\subset \ker \epsilon$. Consequently, 

\begin{equation}
\begin{split}
(\epsilon \otimes \id)\widetilde{\D}(X)&=(\epsilon \otimes \id)(\widetilde{\D}(X)+X'_1\otimes X'_2)\\
&=(\epsilon \otimes \id)\D(X)\\
&=X.
\end{split}
\end{equation}
A symmetric computation gives $(\id \otimes \epsilon)\widetilde{\D}(X)=X$. Hence we get $(\epsilon \otimes \id)\widetilde{\D}=(\id \otimes \epsilon)\widetilde{\D}=\id$.

We complete the proof of the proposition.
\end{proof}
 
 Using Proposition \ref{p3.6}, we can now derive an interesting and important result. We need the following two lemmas. The first generalizes \cite[Corollary 5.2.12]{HS24} and is proved by a similar method.  
 
\begin{lemma}\label{l3.7}
Let $H=\oplus_{n\geq 0} H_n$ be a connected graded braided bialgebra in ${_{\k G}^{\k G} \mathcal{YD}^\Phi}$ that is generated by $H_1$. Then $H$ is a braided Hopf algebra.
\end{lemma} 
 \begin{proof}
It suffices to show that $H$ admits an antipode. Let $\eta$ and $\varepsilon$ be the unit and counit of $H$ respectively, and let $$\mathcal{S}=\sum_{l\geq 0}(\eta\varepsilon-\id)^{*l},$$ where $*$ denotes the convolution product and we set $(\eta\varepsilon-\id)^{*0}=\eta\varepsilon$.
Let $1$ be the identity element of $H$. Since $(\eta\varepsilon-\id)(1)=0$, for any homogenous element $X\in H_l$ we have $(\eta\varepsilon-\id)^{*m}(X)=0$ whenever $m>l$. Hence $\mathcal{S}(X)=\sum_{1\leq i\leq l}(\eta\varepsilon-\id)^{*i} (X)$ is a finite sum, and $\mathcal{S}$ is well defined linear endomorphism of $H$. Set $f=\eta\varepsilon-\id$, we have
\begin{equation*}
\begin{split}
(\id*\mathcal{S})(X)&=((\eta\varepsilon-f)* \sum_{1\leq i\leq l}f^{*i})(X)\\ 
&=\Big(\eta\varepsilon({X_1})-f({X_1})\Big)\Big(\sum_{0\leq i\leq l}f^{*i}({X_2})\Big)\\
&= \sum_{0\leq i\leq l}f^{*i}({X})- \sum_{0\leq i\leq l}f^{*i+1}({X})\\
&=\eta\varepsilon(X).
\end{split}
\end{equation*}
Here the last equality uses the fact that $f^{*l+1}({X})=0$. Thus $\id* \mathcal{S}=\eta\varepsilon$. A symmetric argument gives $\mathcal{S}*\id=\eta\varepsilon$. Therefore $\mathcal{S}$ is an antipode for $H$, and consequently $H$ is a braided Hopf algebra. 
\end{proof}

\begin{lemma}\label{l3.8}
Let $V=V_1\oplus \cdots \oplus V_\theta\in {_{\k G}^{\k G} \mathcal{YD}^\Phi}$ be a direct sum of simple objects with $\theta\geq 3$, and let $\mathcal{P}(V)$ be a graded pre-Nichols algebra of $V$ with $\deg(V_i)=e_i$ for $1\leq i\leq \theta$. Let $U=\ad_{V_i}(\ad_{V_j}(V_k))\subset \mathcal{P}(V)$, where the indices $i,j,k\in \{1,2,\ldots, \theta\}$ are pairwise distinct. If $U\neq 0$, then for every nonzero element $\mathbbm{X}\in U$ the coproduct takes the form 
$$\D(\mathbbm{X})=1\otimes \mathbbm{X}+ \mathbbm{X}\otimes 1+\sum_{l\in \mathcal{J}} {^l\mathbbm{X}}_{1}\otimes {^l\mathbbm{X}}_{2},$$
where for all $l\in \mathcal{J}$ (if $\mathcal{J}\neq \emptyset$) we have 
\begin{equation}\label{eq3.14}
\deg({^l\mathbbm{X}}_{1})\in \{e_i,e_j,e_i+e_j\},\ \ \ \deg({^l\mathbbm{X}}_{2})\in \{e_i+e_k,e_j+e_k,e_k\}.
\end{equation}
\end{lemma}

\begin{proof}
Assume the $G$-degree of $V_l$ is $g_l\in G$ for $1\leq l\leq \theta$. Without loss of generality, write $\mathbbm{X}=\ad_X(\ad_Y(Z))$ with $X\in V_i, Y\in V_j, Z\in V_k$. Direct calculation yields 
\begin{equation}\label{eq3.22}
\begin{split}
\D(\ad_Y(Z))&=\D(YZ -g_j \triangleright ZY)\\
&=(Y\otimes 1+1\otimes Y)(Z\otimes 1+1\otimes Z)\\
&\ \ \ \ -(g_j \triangleright Z\otimes 1+1\otimes g_j \triangleright Z)(Y\otimes 1+1\otimes Y)\\
&=\ad_Y(Z)\otimes 1+1\otimes \ad_Y(Z)+Y\otimes Z-g_k\triangleright Y\otimes g_j \triangleright Z,
\end{split}
\end{equation}
and 
\begin{equation}
\begin{split}
&\D(g_i\triangleright \ad_Y(Z))=g_i\triangleright \D(\ad_Y(Z))\\
&=g_i\triangleright (\ad_Y(Z)\otimes 1+1\otimes \ad_Y(Z)+Y\otimes Z-g_k\triangleright Y\otimes g_j \triangleright Z)\\
&=g_i\triangleright \ad_Y(Z)\otimes 1+1\otimes g_i\triangleright \ad_Y(Z)\\
&\ \ \ +\Phi_{g_i}(g_j,g_k)g_i\triangleright Y\otimes g_i\triangleright Z- \Phi_{g_i}(g_j,g_k)g_i\triangleright (g_k\triangleright Y) \otimes g_i\triangleright (g_j \triangleright Z)\\
&=g_i\triangleright \ad_Y(Z)\otimes 1+1\otimes g_i\triangleright \ad_Y(Z)+\Phi_{g_i}(g_j,g_k)g_i\triangleright Y\otimes g_i\triangleright Z\\
&\ \ \ -\Phi_{g_i}(g_j,g_k)\Phi_{g_j}(g_i,g_k)\Phi_{g_k}(g_i,g_j) (g_ig_k)\triangleright Y \otimes (g_ig_j) \triangleright Z.
\end{split}
\end{equation}

Consequently, we have 
\begin{equation}\label{eq3.24}
\begin{split}
&\D(\ad_X(\ad_Y(Z)))\\
&=\D(X\ad_Y(Z)-g_i\triangleright \ad_Y(Z) X)\\
&=(X\otimes 1+1\otimes X)\\
&\ \ \ \ \ \ \ \times \big[\ad_Y(Z)\otimes 1+1\otimes \ad_Y(Z)+Y\otimes Z-g_k\triangleright Y\otimes g_j \triangleright Z\big]\\
&\ \  - \big[g_i\triangleright \ad_Y(Z)\otimes 1+1\otimes g_i\triangleright \ad_Y(Z)+\Phi_{g_i}(g_j,g_k)g_i\triangleright Y\otimes g_i\triangleright Z\\
&\ \ \ \ \ \ \ \ \ -\Phi_{g_i}(g_j,g_k)\Phi_{g_j}(g_i,g_k)\Phi_{g_k}(g_i,g_j) (g_ig_k)\triangleright Y \otimes (g_ig_j) \triangleright Z\big]\\
&\ \ \ \ \ \ \ \times (X\otimes 1+1\otimes X)\\
&=1\otimes \ad_X(\ad_Y(Z))+\ad_X(\ad_Y(Z))\otimes 1\\
& \ \ + \big[X\otimes \ad_Y(Z)-(g_jg_k)\triangleright X\otimes g_i\triangleright \ad_Y(Z)\big]\\
&\ \ +\big[\Phi(g_i,g_j,g_k)\Phi^{-1}(g_j,g_i,g_k)g_i\triangleright Y\otimes \ad_X Z \\
& \ \ \ \ \ \ \ \ \ \ \ \ \ -\Phi(g_i,g_k,g_j)g_ig_k\triangleright Y\otimes \ad_X(g_j \triangleright Z) \big]\\
&\ \ +\Phi(g_i,g_j,g_k)\big[ XY\otimes Z-Xg_k\triangleright Y\otimes g_j\triangleright Z- g_i\triangleright Y g_k\triangleright X\otimes g_i\triangleright Z\\
&\ \ \ \ \ \ \ \ \ \ \ \ \ \ \ \ \ \ \ \ \ \ +\Phi(g_j,g_i,g_k)\Phi(g_k,g_i,g_j)g_ig_k\triangleright Yg_k\triangleright X\otimes g_ig_j\triangleright Z\big].
\end{split}
\end{equation}

Define  
\begin{eqnarray}
{^1\mathbbm{X}}_{1}\otimes {^1\mathbbm{X}}_{2}&=&X\otimes \ad_Y(Z)-(g_jg_k)\triangleright X\otimes g_i\triangleright \ad_Y(Z),
\end{eqnarray}
\begin{equation}
\begin{split}
{^2\mathbbm{X}}_{1}\otimes {^2\mathbbm{X}}_{2}& \ \ =\ \ \Phi(g_i,g_j,g_k)\Phi^{-1}(g_j,g_i,g_k)g_i\triangleright Y\otimes \ad_X Z \\
& \ \ \ \ \ \  -\Phi(g_i,g_k,g_j)g_ig_k\triangleright Y\otimes \ad_X(g_j \triangleright Z),
\end{split}
\end{equation}
and 
\begin{equation}
\begin{split}
&{^3\mathbbm{X}}_{1}\otimes {^3\mathbbm{X}}_{2}\\
&=\Phi(g_i,g_j,g_k)\big[ XY\otimes Z-Xg_k\triangleright Y\otimes g_j\triangleright Z- g_i\triangleright Y g_k\triangleright X\otimes g_i\triangleright Z\\
&\ \ \ \  +\Phi(g_j,g_i,g_k)\Phi(g_k,g_i,g_j)g_ig_k\triangleright Yg_k\triangleright X\otimes g_ig_j\triangleright Z\big].
\end{split}
\end{equation}
Set $\mathcal{J}=\{l|{^l\mathbbm{X}}_{1}\otimes {^l\mathbbm{X}}_{2}\neq 0,1\leq l\leq 3\}$. Then $$\D(\mathbbm{X})=1\otimes \mathbbm{X}+ \mathbbm{X}\otimes 1+\sum_{l\in \mathcal{J}} {^l\mathbbm{X}}_{1}\otimes {^l\mathbbm{X}}_{2},$$ 
and whenever $\mathcal{J}\neq \emptyset$ it is clear that ${^l\mathbbm{X}}_{1}$ and ${^l\mathbbm{X}}_{2}$ satisfy condition \eqref{eq3.14}.
\end{proof}

\begin{proposition}\label{p3.11}
Let $V=V_1\oplus \cdots \oplus V_\theta\in {_{\k G}^{\k G} \mathcal{YD}^\Phi}$ be a direct sum of simple objects with $\theta\geq 3$, and let $\mathcal{P}(V)$ be a graded pre-Nichols algebra of $V$ with $\deg(V_i)=e_i$ for $1\leq i\leq \theta$. Define $W=\ad_{V_i}(\ad_{V_j}(V_k))+\ad_{V_j}(\ad_{V_i}(V_k)) \subset \mathcal{P}(V)$, where the indices $i,j,k\in \{1,2,\ldots, \theta\}$ are pairwise distinct,  and let $A(W)$ be the subalgebra of $\mathcal{P}(V)$ generated by $W$. If $W$ is nonzero, then $A(W)$ is a pre-Nichols algebra of $W$. 
\end{proposition}
\begin{proof}
Suppose that $W\neq 0$, then $\Deg(W)=\{e_i+e_j+e_k\}$. Let $\mathbbm{X}\in W$ be a nonzero element. By Lemma \ref{l3.8} we have 
\begin{equation}\label{eq3.19}
\D(\mathbbm{X})=1\otimes \mathbbm{X}+ \mathbbm{X}\otimes 1+\sum_{l\in \mathcal{J}} {^l\mathbbm{X}}_{1}\otimes {^l\mathbbm{X}}_{2},
\end{equation}
where $\deg({^l\mathbbm{X}}_{1})\in \{e_i,e_j,e_i+e_j\}$ and $\deg({^l\mathbbm{X}}_{2})\in \{e_i+e_k,e_j+e_k,e_k\}$ for each $l\in \mathcal{J}$ if $\mathcal{J}\neq \emptyset$.

Define
\begin{equation*}
\begin{split}
&S=\{k(e_i+e_j+e_k)+ae_i+be_j|k,a,b\in \N_0, a+b\neq 0\},\\
&T=\{k(e_i+e_j+e_k)+ae_i+be_j+ce_k|k,a,b,c\in \N_0, ab=0, c>0, a+b\leq c\}.
\end{split}
\end{equation*}
Clearly, both $S$ and $T$ are additively closed, and $ \{e_i,e_j,e_i+e_j\}\subset S$, $\{e_i+e_k,e_j+e_k,e_k\}\subset T$. Now, set 
\begin{equation}
\begin{split}
&M=\{X\in \mathcal{P}(V)|\deg (X)\in S\}\\
&N=\{X\in \mathcal{P}(V)|\deg (X)\in T\}.
\end{split}
\end{equation}
Because $S$ and $T$ are closed under addition, we have $M^2\subset M$ and $N^2\subset N$. Since $$\Deg(A(W))=\{n(e_i+e_j+e_k)| n\in \N_0\},$$ it is immediate that $\Deg(A(W))\cap \Deg(M)=\emptyset, Deg(A(W))\cap \Deg(N)=\emptyset.$
Moreover, because $\Deg(A(W))+S\subset S$ and  $\Deg(A(W))+T\subset T$, the spaces $M$ and $N$ are $A(W)$-bimodules under the product of $\mathcal{P}(V)$.

We now show that 
$$\D(A(W))\subset A(W) \otimes A(W)+M\otimes N.$$ 
From \ref{eq3.19}, we already have $\D(W)\subset A(W) \otimes A(W)+M\otimes N$. Assume inductively that $\D(W^{i})\subset A(W) \otimes A(W)+M\otimes N$ for all $1\leq i\leq n-1$. Then we have 
\begin{equation*}
\begin{split}
\D(W^{n})&=\D(W)\D(W^{n-1})\subset (A(W) \otimes A(W)+M\otimes N) (A(W) \otimes A(W)+M\otimes N)\\
&\subset  (A(W) \otimes A(W)+M\otimes N).
\end{split}
\end{equation*}
Since $A(W)$ is generated by $W$, it follows that $\D(A(W))\subset A(W) \otimes A(W)+M\otimes N$. 

Applying Proposition \ref{p3.6}, we obtain that $A(W)$ becomes a braided bialgebra in  ${_{\k G}^{\k G} \mathcal{YD}^\Phi}$ under the new comultiplication $\widetilde{\D}=P\circ \D$, where $P:A\otimes A+ M\otimes N \to A\otimes A$ is the canonical projection. Finally, $A(W)$ is connected graded braided bialgebra generated by $W$, and every nonzero element in $W$ is primitive under $\widetilde{\D}$. Thus, by Lemma \ref{l3.7}, $A(W)$ is a pre-Nichols algebras.  
\end{proof}

\section{Nichols algebras of diagonal type}
In this section, we classify finite GK-dimensional Nichols algebras of diagonal type in the category $_{\k G}^{\k G} \mathcal{YD}^\Phi$. We also provide a classification of finite GK-dimensional Nichols algebras arising from simple objects. These results serve as a foundation for the further study of finite GK-dimensional pre-Nichols algebras.

Recall that an ordinary braided vector space is a pair $(V, c)$, where $V$ is a vector space and  $c\in GL(V\otimes V)$
satisfies the braid equation:
\begin{equation}\label{eq4.1}
(\id\otimes c)(c\otimes \id)(\id\otimes c)=(c\otimes \id)(\id\otimes c)(c\otimes \id).
\end{equation}
For an object $V$ in the category $_{\k G}^{\k G} \mathcal{YD}^\Phi$, its braiding $c$ does not generally satisfy \eqref{eq4.1}. Instead, it fulfills the following twisted analogue:
\begin{equation}\label{eq4.2}
\begin{split}
&(\id\otimes c)\circ a_{V,V,V}\circ (c\otimes \id)\circ a^{-1}_{V,V,V}\circ (\id\otimes c)\circ a_{V,V,V}\\
& =a_{V,V,V}\circ(c\otimes \id)\circ a^{-1}_{V,V,V}\circ (\id\otimes c)\circ a_{V,V,V}\circ (c\otimes \id),
\end{split}
\end{equation}
where $a$ denotes the associativity constraint of $_{\k G}^{\k G} \mathcal{YD}^\Phi$, determined by the $3$-cocycle $\Phi$ (see definition \ref{d2.1}). Consequently, when $\Phi$ is nontrivial, an object $V$ in the category $_{\k G}^{\k G} \mathcal{YD}^\Phi$ is not an ordinary braided vector space.

In \cite{AI22}, Angiono and Garc\'ia Iglesias classified all finite GK-dimensional Nichols algebras of ordinary braided vector spaces of diagonal type.
 \begin{theorem}\cite[Theorem 1.2]{AI22}\label{t4.1}
Let $V$ be a finite-dimensional braided vector space of diagonal type. Then $B(V)$ is finite GK-dimensional if and only if the corresponding root system is finite, that is, an arithmetic root system.
\end{theorem}

Applying the theorem, we obtain a classification of finite GK-dimensional Nichols algebras of diagonal type in the category ${_{\k G}^{\k G} \mathcal{YD}^\Phi}$. 
\begin{theorem}\label{t4.2}
Let $V\in {_{\k G}^{\k G} \mathcal{YD}^\Phi}$ be a finite-dimensional object of diagonal type. Then $B(V)$ has finite GKdim if and only if the corresponding root system is finite, i.e., it coincides with an arithmetic root system.
\end{theorem}

\begin{proof}
Let $G_V$ be the support group of $V$. Because $V$ is of diagonal type, Lemma \ref{l2.4} implies that $\Phi_{G_V}$ is an abelian $3$-cocycle on $G_V$. By Lemma \ref{l3.1}, $B(V)$ is a Nichols algebra in 
$_{\k G_V}^{\k G_V} \mathcal{YD}^{\Phi_{G_V}}$. Let $\pi:\widehat{G_V} \to G_V$ be the canonical homomorphism defined in \eqref{eq2.13}. According to Corollary \ref{c3.5}, $B(V)$ is twist equivalent to a Nichols algebra $B(V^J)$ in the ordinary Yetter-Drinfeld category $_{\k \widehat{G_V}}^{\k \widehat{G_V}} \mathcal{YD}$, where $J$ is a $2$-cochain on $\widehat{G_V}$ satisfying $\partial J=(\pi^*\Phi_{G_V})^{-1}$.

Assume that $B(V)$ has finite GKdim. Then, by Lemma \ref{l2.11}, $B(V^J)$ also has finite GKdim. Therefore Theorem \ref{t4.1} implies that the root system $\bigtriangleup(B(V^J))$ is finite. Since twisting does not change the root system by Proposition \ref{p2.10}, we have $\bigtriangleup(B(V))=\bigtriangleup(B(V^J))$, consequently $\bigtriangleup(B(V))$ is finite. Conversely, if $\bigtriangleup(B(V))$ is finite, then $\bigtriangleup(B(V^J))$ is finite as well, and Theorem \ref{t4.1} yields that $\GKdim(B(V^J))$ is finite. Since twisting does not change GKdim, $\GKdim(B(V))$ is also finite. This complete the proof. 
\end{proof}

Let $V\in {_{\k G}^{\k G} \mathcal{YD}^\Phi}$ be a finite-dimensional object of diagonal type. Recall that the generalized Dynkin diagram $\mathcal{D}(V)$ is said to be of finite type when the root system $\bigtriangleup(B(V))$ is finite. An immediate consequence of Theorem \ref{t4.2} is the following.

\begin{corollary}\label{c4.3}
Let $V\in {_{\k G}^{\k G} \mathcal{YD}^\Phi}$ be a finite-dimensional object of diagonal type such that $B(V)$ has finite GKdim. Then $\mathcal{D}(V)$  is of finite type.
\end{corollary}

In \cite{H4}, Heckenberger gave a complete list of generalized Dynkin diagrams of finite type for Nichols algebras of diagonal type in $_{\k  \mathbbm{G}}^{\k  \mathbbm{G}} \mathcal{YD}$, where $\mathbbm{G}$ is an abelian group. In what follows, we denote by $\mathfrak{D}$ the set of generalized Dynkin diagrams of finite type that appear in that list. 
\begin{proposition}\label{r4.4}
Let $V\in {_{\k G}^{\k G} \mathcal{YD}^\Phi}$ be a finite-dimensional object of diagonal type. If the generalized Dynkin diagram $\mathcal{D}(V)$ is of finite type, then $\mathcal{D}(V)$ belongs to $\mathfrak{D}$. 
\end{proposition}
\begin{proof}
Because $V$ is of diagonal type, Proposition \ref{p3.6} show that $B(V)$ is twist equivalent to a Nichols algebra $B(V^J)$ in the ordinary Yetter-Drinfeld module category ${_{\k \widehat{G_V}}^{\k \widehat{G_V}} \mathcal{YD}}$.
Assume that $\mathcal{D}(V)$ is of finite type, then $\D(B(V))$ is finite. By Proposition \ref{p2.10}, the root system remains unchanged under twisting, so $\D(B(V^J))$ is also finite. Hence $\mathcal{D}(V^J)$ is of finite type, and therefore belongs to $\mathfrak{D}$. Applying Proposition \ref{p2.10} once more, the diagrams $\mathcal{D}(V)$ and $\mathcal{D}(V^J)$ coincide. Thus $\mathcal{D}(V)$ belongs to $\mathfrak{D}$. 
\end{proof}

A generalized Dynkin diagram is called an $n$-cycle if it has vertices $\{1,2,\ldots, n\}$ and two vertices $i, j$ are connected precisely when $|i-j|=1$ or $\{i, j\}=\{1,n\}$. The following corollary plays an important role in the study of GKdim of Nichols algebras.
\begin{corollary}\label{c4.5}
Let $V\in {_{\k G}^{\k G} \mathcal{YD}^\Phi}$ be an object of diagonal type. If the generalized Dynkin diagram $\mathcal{D}(V)$ contains an $n$-cycle with $n\geq 4$, then $\GKdim(B(V))=\infty$.
\end{corollary}
\begin{proof}
Assume that $\mathcal{D}(V)$ contains an $n$-cycle with $n\geq 4$. Comparing with the set $\mathfrak{D}$ of generalized Dynkin diagrams of finite type listed in \cite{H4}, we see that $\mathcal{D}(V)$ does not belong to $\mathfrak{D}$, hence it is not of finite type. By Corollary \ref{c4.3}, we obtain $\GKdim(B(V))=\infty$.
\end{proof}

Since every simple object in ${_{\k G}^{\k G} \mathcal{YD}^\Phi}$ is necessarily of diagonal type, we can apply Theorem \ref{t4.2} to classify all finite GK-dimensional Nichols algebras of simple objects.
\begin{theorem}\label{t4.6}
Let $V\in {_{\k G}^{\k G} \mathcal{YD}^\Phi}$ be a simple object with $\dim(V)\geq 2$, and let $g=g_V$. Then $B(V)$ has finite GKdim if and only if $V$ is one of the following types:
\begin{itemize}
\item[(T1)] $g\triangleright v=v$ for all $v\in V$.
\item[(T2)] $g\triangleright v=-v$ for all $v\in V$.
\item[(T3)] $g\triangleright v=\zeta_3 v$ for all $v\in V$ and $\dim(V)=2$, where $\zeta_3$ is a primitive third root of unity. 
\end{itemize}
\end{theorem}
\begin{proof}
Since $V$ is simple, the action of its $G$-degree element $g$ is given by $g\triangleright v=\zeta v$ for all $v\in V$, where $\zeta\in \k^*$ is a scalar. Let $n=\dim(V)$ and fix a basis $\{x_1,\ldots, x_n\}$ of $V$. If $\zeta=1$, then $B(V)$ is isomorphic to the symmetric algebra $S(V)$, whence $\GKdim(B(V))=\dim(V)$. If $\zeta=-1$, then $B(V)=T(V)/I$ with $$I=\langle X_iX_j+X_jX_i, X_i^2|1\leq i\neq j\leq n\rangle.$$ So $B(V)$ is a finite dimensional Nichols algebra and hence $\GKdim(B(V))=0$.

Now assume that $\zeta\neq \pm 1$. If $n\geq 3$, then the generalized Dynkin diagram $\mathcal{D}(V)$ contains an $n$-cycle whose vertices are labelled by $\zeta$ and whose edges are labelled by $\zeta^2$. Comparing with the set $\mathfrak{D}$ of generalized Dynkin diagrams of finite type, we see that such a diagram does not appear in $\mathfrak{D}$. Therefore it is not of finite type, and consequently $\GKdim(B(V))=\infty$.

If $n=2$, the diagram $\mathcal{D}(V)$ has the form   
\[ {\setlength{\unitlength}{1.5mm}
\Dchaintwo{}{$\zeta$}{$\zeta^2$}{$\zeta$}} \quad.\]
A direct inspection of the list $\mathfrak{D}$ shows that the only diagram of this shape that is of finite type is 
\[ {\setlength{\unitlength}{1.5mm}
\Dchaintwo{}{$\zeta_3$}{$\zeta_3^{-1}$}{$\zeta_3$}} \quad,\]
where $\zeta_3$ is a primitive third root of unity. In this case, $\mathcal{D}(V)$ is of Cartan type $A_2$ with parameter $q=\zeta_3$, the Nichols algebra $B(V)$ is finite-dimensional and thus $\GKdim(B(V))=0$.  
\end{proof}

\begin{remark}
Let $V\in {_{\k G}^{\k G} \mathcal{YD}^\Phi}$ be a simple object with $\dim(V)\geq 2$. Then $B(V)$ is finite-dimensional precisely when $V$ is of type (T2) or (T3). This result has already been established in \cite{HYZ}.
\end{remark}

\section{Minimal nondiagonal objects} 

In this section, we initiate the study of pre-Nichols algebras of nondiagonal type, where the so-called minimal nondiagonal objects play a central role. Recall that an object $V\in {_{\k G}^{\k G} \mathcal{YD}^\Phi}$ is said to be minimal nondiagonal if it is nondiagonal and every nonzero proper subobject of $V$ is diagonal. We shall first analyze the Yetter-Drinfeld module structure of minimal nondiagonal objects. Afterward, we prove that if $V$ is a minimal nondiagonal object satisfying $G_V=G$ and $\dim(V)\geq 9$, then $\GKdim(B(V))=\infty$.

Let $V=V_1\oplus \cdots \oplus V_\theta \in {_{\k G}^{\k G} \mathcal{YD}^\Phi}$ be a direct sum of $\theta$ simple objects. If $\theta=1$ or $2$, the support group of $V$ is necessarily cyclic or a direct product of two cyclic groups, this forces $V$ to be of diagonal type by Lemma \ref{l2.2} and Lemma \ref{l2.4}. Hence, if $V$ is nondiagonal, we must have $\theta \geq 3$. This observation yields the following statement.
\begin{lemma}\label{l5.1}
Every nondiagonal object in ${_{\k G}^{\k G} \mathcal{YD}^\Phi}$ must be a direct sum of at least $3$ simple objects.
\end{lemma}
We now turn to the structure of minimal nondiagonal objects. The following result will be needed. 

\begin{lemma}\cite[Lemma 4.5]{HYZ}\label{l5.2} 
Suppose that the set $\{g_1,g_2,\ldots, g_\theta\}$ generates the group $G$ and that $$\Phi_{g_i}(g_j,g_k)=\Phi_{g_i}(g_k,g_j)$$ for all $g_i,g_j,g_k\in \{g_1,g_2,\ldots, g_\theta\},$ then $\Phi$ is an abelian $3$-cocycle on $G$.
\end{lemma}

\begin{lemma}\label{l5.3}
Let $V=V_1\oplus\cdots \oplus V_\theta \in {_{\k G}^{\k G} \mathcal{YD}^\Phi}$ be a direct sum of simple objects. If $V$ is nondiagonal, then there exist indices $i_1,i_2,i_3\in \{1,2,\ldots, \theta\}$ such that the subobject $V_{i_1}\oplus V_{i_2}\oplus V_{i_3}$ is nondiagonal.
\end{lemma}
\begin{proof}
Assume, to the contrary, that $V_{i_1}\oplus V_{i_2}\oplus V_{i_3}$ is diagonal for every triple $i_1,i_2,i_3\in \{1,2,\ldots, \theta\}$. Choose bases $\{X_1,\ldots, X_r\}$,  $\{Y_1,\ldots, Y_s\}$, $\{Z_1,\ldots, Z_t\}$ of $V_{i_1},V_{i_2}$, $V_{i_3}$, respectively , such that their union is a diagonal basis of $V_{i_1}\oplus V_{i_2}\oplus V_{i_3}$. Let $g_i=g_{V_i}$, $1\leq i\leq \theta$. For each $1\leq k\leq r$ we have 
$$g_{i_2}\triangleright  X_k=p_k X_k, \ \ g_{i_3}\triangleright  X_k=q_k X_k$$
for certain constants $p_k,q_k\in \k^*$. Then
\begin{eqnarray*}
&\Phi_{g_{i_1}}(g_{i_2}, g_{i_3})g_{i_2}g_{i_3}\triangleright  X_k=g_{i_2}\triangleright(g_{i_3}\triangleright  X_k)=p_kq_k X_k,\\
&\Phi_{g_{i_1}}(g_{i_3}, g_{i_2})g_{i_2}g_{i_3}\triangleright  X_k=g_{i_3}\triangleright(g_{i_2}\triangleright  X_k)=p_kq_k X_k.
\end{eqnarray*} 
Thus $\Phi_{g_{i_1}}(g_{i_2}, g_{i_3})=\Phi_{g_{i_1}}(g_{i_3}, g_{i_2})$. Since this equality holds for all triples of indices, Lemma \ref{l5.2} implies that $\Phi_{G_V}$ is an abelian $3$-cocycle on $G_V$. By Lemma \ref{l2.4}, the whole object $V=V_1\oplus\cdots \oplus V_\theta$ would be diagonal, which is a contradiction to the hypothesis that $V$ is nondiagonal.
\end{proof}
We can now give a precise characterisation of minimal nondiagonal objects. 
\begin{proposition}\label{p5.4}
Let $V\in {_{\k G}^{\k G} \mathcal{YD}^\Phi}$ be a direct sum of $\theta$ simple objects. Then $V$ is a minimal nondiagonal object if and only if $\theta=3$ and the restricted $3$-cocycle $\Phi_{G_V}$ is nonabelian on the support group $G_V$.
\end{proposition}
\begin{proof}
Let $V=V_1\oplus\cdots \oplus V_\theta$ is a direct sum of simple objects. Assume $V$ is a minimal nondiagonal object. Lemma \ref{l5.3} implies that $\theta\leq 3$, while Lemma \ref{l5.1} gives $\theta\geq 3$. Hence $\theta=3$ and we can write $V=V_1\oplus V_2 \oplus V_3$. By Lemma \ref{l2.4}, $\Phi_{G_V}$ is a nonabelian $3$-cocycle on $G_V$ since $V$ is nondiagonal.

Conversely, suppose $V=V_1\oplus V_2 \oplus V_3$ is a direct sum of  three simple objects and $\Phi_{G_V}$ is a nonabelian $3$-cocycle on $G_V$. Then Lemma \ref{l2.4} implies that $V$ is nondiagonal. Moreover, by Lemma \ref{l5.1}, every nonzero proper subobject of $V$ must be diagonal. Therefore $V$ is minimal nondiagonal.
\end{proof}

The following corollary follows directly from Lemma \ref{l5.3} and Proposition \ref{p5.4}.
\begin{corollary}\label{cor5.5}
\begin{itemize}
\item[(1)] Let $V\in {_{\k G}^{\k G} \mathcal{YD}^\Phi}$ be a minimal nondiagonal object with $G_V=G$. Then $G$ can be generated by three elements. 
\item[(2)] Every nondiagonal object in ${_{\k G}^{\k G} \mathcal{YD}^\Phi}$ contains at least one minimal nondiagonal subobject.
\end{itemize}
\end{corollary}

In our earlier work \cite{HLYY3}, we studied the structure of twisted Yetter-Drinfeld modules $V\in {_{\k G}^{\k G} \mathcal{YD}^\Phi}$ that decompose into a direct sum of three simple objects and satisfy $G_V=G$. This study essentially characterizes the structure of minimal nondiagonal objects. 

\begin{proposition}\cite[Lemma 3.5, Proposition 3.6]{HLYY3}\label{p5.6} 
Let $V=V_1\oplus V_2\oplus V_3\in {_{\k G}^{\k G} \mathcal{YD}^\Phi}$ be a direct sum of simple objects with $G_V=G$. Denote $g_i=g_{V_i}\in G$ and $m_i=|g_i|$ for $1\leq i\leq 3$. Then the following hold:
\begin{itemize}
\item[(1)] $\dim(V_1)=\dim(V_2)=\dim(V_3)=n$, where $n$ is the order of $\frac{{\Phi}_{g_1}(g_2,g_3)}{{\Phi}_{g_1}(g_3,g_2)}$. 
\item[(2)] If $V$ is nondiagonal, then $n> 1$, and for a fixed $i\in \{1,2,3\}$ there exists a basis $\{X_1,X_2,\ldots, X_n\}$ of $V_i$ such that 
\begin{eqnarray}
&&g_i\triangleright X_l=\alpha_i X_l,\  \ 1\leq l\leq n;\label{e6.5}\\
&&g_j\triangleright X_l=\beta_i (\frac{{\Phi}_{g_i}(g_j,g_k)}{{\Phi}_{g_i}(g_k,g_j)})^{l-1} X_l,\ \ 1\leq l\leq n; \label{e6.6}\\
&&g_k\triangleright X_l=X_{l+1}, \ g_k\triangleright X_n=\gamma_i X_1, \  \ 1\leq l\leq n-1.\label{e6.7}
\end{eqnarray}
Here $j\neq k\in \{1,2,3\}\setminus \{i\}$, and the scalars $\alpha_i,\beta_i,\gamma_i\in \k^*$ satisfy
\begin{eqnarray}
\alpha_i^{m_i}&=&\prod_{l=1}^{m_i-1}\Phi_{g_i}(g_i, g_i^l),\\
\beta_i^{m_j}&=&\prod_{l=1}^{m_j-1}\Phi_{g_i}(g_j, g_j^l), \\
\gamma_i^{\frac{m_k}{n}}&=&\prod_{l=1}^{m_k-1}\Phi_{g_i}(g_k, g_k^l).
\end{eqnarray}
\end{itemize}
\end{proposition}

\begin{corollary}\label{c5.7}
Let $V$ be a minimal nondiagonal object in ${_{\k G}^{\k G} \mathcal{YD}^\Phi}$ satisfying $G_V=G$. Then $\dim(V)\geq 6$ and each simple subobject of $V$ has dimension $\frac{1}{3}\dim(V)$.
\end{corollary}

It is worth noting that in Proposition \ref{p5.6}, the dimensions of $V_1,V_2$ and $V_3$ are also equal to the order of either $\frac{{\Phi}_{g_2}(g_3,g_1)}{{\Phi}_{g_2}(g_1,g_3)}$ or $\frac{{\Phi}_{g_3}(g_1,g_2)}{{\Phi}_{g_3}(g_2,g_1)}$. This equivalence follows from the following identities (see \cite[Lemma 3.4]{HLYY3}):
 \begin{equation}\label{eq5.7}
 \frac{{\Phi}_{g_1}(g_2,g_3)}{{\Phi}_{g_1}(g_3,g_2)}=\frac{{\Phi}_{g_2}(g_3,g_1)}{{\Phi}_{g_2}(g_1,g_3)}=\frac{{\Phi}_{g_3}(g_1,g_2)}{{\Phi}_{g_3}(g_2,g_1)}.
 \end{equation}
Based on Corollary \ref{c5.7}, we can introduce the following definition.
\begin{definition}
Let $G$ be a finite abelian group that can be generated by $3$-elements, and let $\Phi$ be a nonabelian $3$-cocycle on $G$. For each integer $n\geq 2$, denote $\mathcal{F}_n$ be the set of minimal nondiagonal objects $V\in {_{\k G}^{\k G} \mathcal{YD}^\Phi}$ such that $G_V=G$ and $\dim(V)=3n$.
\end{definition}

Let $V=V_1\oplus V_2\oplus V_3\in \mathcal{F}_n$. If for some $i\in \{1,2,3\}$ the simple object $V_i$ is not of the types (T1)-(T3) in Theorem \ref{t4.6},  then $\GKdim(B(V_i))=\infty$  and hence also $\GKdim(B(V))=\infty$, because $B(V_i)$ is an $\N_0$-graded subalgebra of $B(V)$. Therefore, in the following we shall only consider the case in which $V_1$, $V_2$, $V_3$ are all of types (T1)-(T3).  

For the rest of this paper, let $\mathcal{F}^{T}_n\subset \mathcal{F}_n$ denote the subset 
\begin{equation}
\mathcal{F}^{T}_n=\{V_1\oplus V_2\oplus V_3\in \mathcal{F}_n|\ V_i\ \mathrm{is}\  \mathrm{of}\ \mathrm{type}\ (T1), (T2) \ \mathrm{or}\ (T3), 1\leq i\leq 3\}.
\end{equation}
For each $i\in \{1,2,3\}$, we further set
\begin{equation}
\mathcal{F}_n^{(Ti)}=\{V_1\oplus V_2\oplus V_3\in \mathcal{F}_n| V_1,V_2,V_3\ \mathrm{are}\ \mathrm{all}\ \mathrm{of}\ \mathrm{type}\ (Ti)\}.
\end{equation}

Throughout the remainder of this section, we assume that $G$ can be generated by $3$-elements and that $\Phi$ is a nonabelian $3$-cocycle on $G$. We now prove that if $V$ is a minimal nondiagonal object satisfying $G_V=G$ and $\dim(V)\geq 9$, then $\GKdim(B(V))=\infty$.

\begin{proposition}\label{p5.9}
Assume $n\geq 3$ and $\mathcal{F}_n\neq \emptyset$. Then $\GKdim(B(V))=\infty$ for all $V\in \mathcal{F}_n$.
\end{proposition} 
 
\begin{proof}
Let $V=V_1\oplus V_2\oplus V_3\in \mathcal{F}_n$, and denote $g_i=g_{V_i}$ for $1\leq i\leq 3$. By the definition of $\mathcal{F}_n$, we have 
$G=G_V=\langle g_1,g_2,g_3\rangle$ and $n=|\frac{{\Phi}_{g_1}(g_2,g_3)}{{\Phi}_{g_1}(g_3,g_2)}|\geq 3.$
We shall prove that $\GKdim(B(V_1\oplus V_2))=\infty$, which forces $\GKdim(B(V))=\infty$ because $B(V_1\oplus V_2)$ is an $\N_0$-graded subalgebra of $B(V)$.

Set $m_i=|g_i|$ for $1\leq i\leq 3$. According to Proposition \ref{p5.6}, there exists a basis $\{X_1,X_2, \ldots, X_n\}$ of $V_1$ such that 
\begin{eqnarray}
&&g_1\triangleright X_l=\alpha_1 X_l,\  \ 1\leq l\leq n;\label{e5.8}\\
&&g_2 \triangleright X_l=\beta_1 (\frac{{\Phi}_{g_1}(g_2,g_3)}{{\Phi}_{g_1}(g_3,g_2)})^{l-1} X_l,\ \ 1\leq l\leq n; \label{e5.9}\\
&& g_3 \triangleright X_l=X_{l+1}, \ g_3\triangleright X_n=\gamma_1 X_1, \  \ 1\leq l\leq n-1.\label{e5.10}
\end{eqnarray}
Here $\alpha_1, \beta_1,\gamma_1\in \k^*$ satisfy $$\alpha_1^{m_1}=\prod_{l=1}^{m_1-1}\Phi_{g_1}(g_1, g_1^l),\ 
\beta_1^{m_2}=\prod_{l=1}^{m_2-1}\Phi_{g_1}(g_2, g_2^l),\
\gamma_1^{\frac{m_3}{n}}=\prod_{l=1}^{m_3-1}\Phi_{g_1}(g_3, g_3^l).$$
Similarly, there exists a basis $\{Y_1,Y_2,\ldots, Y_n\}$ of $V_2$ such that 
\begin{eqnarray}
&&g_2\triangleright Y_l=\alpha_2 Y_l,\  \ 1\leq l\leq n;\label{e5.11}\\
&&g_1 \triangleright Y_l=\beta_2 (\frac{{\Phi}_{g_2}(g_1,g_3)}{{\Phi}_{g_2}(g_3,g_1)})^{l-1} Y_l,\ \ 1\leq l\leq n; \label{e5.12}\\
&& g_3 \triangleright Y_l=Y_{l+1}, \ g_3\triangleright Y_n=\gamma_2 Y_1, \  \ 1\leq l\leq n-1.\label{e5.13}
\end{eqnarray}
Here $\alpha_2, \beta_2,\gamma_2\in \k^*$ satisfy $$\alpha_2^{m_2}=\prod_{l=1}^{m_2-1}\Phi_{g_2}(g_2, g_2^l),\
\beta_2^{m_1}=\prod_{l=1}^{m_1-1}\Phi_{g_2}(g_1, g_1^l),\
\gamma_2^{\frac{m_3}{n}}=\prod_{l=1}^{m_3-1}\Phi_{g_2}(g_3, g_3^l).$$

From \eqref{e5.8}-\eqref{e5.9} and \eqref{e5.11}-\eqref{e5.12}, the set $\{X_1,\ldots X_n,Y_1,\ldots, Y_n\}$ forms a diagonal basis of $V_1\oplus V_2$. 
Let $q=\frac{{\Phi}_{g_1}(g_2,g_3)}{{\Phi}_{g_1}(g_3,g_2)}$ and denote by $\mathcal{R}$ the braiding on $V$. Then for all $1\leq s,t\leq n$ we have 
\begin{eqnarray}
\mathcal{R}(X_s\otimes Y_t)&=\beta_2 q^{t-1}Y_t\otimes X_s,\\
\mathcal{R}(Y_t\otimes X_s)&=\beta_1 q^{s-1}X_s\otimes Y_t.
\end{eqnarray}
Consequently, in the generalized Dynkin diagram $\mathcal{D}(V_1\oplus V_2)$, the vertices corresponding to $X_s$ and $Y_t$ are connected precisely when 
$$\beta_1\beta_2 q^{s+t-2}\neq 1.$$
Assume that $\beta_1\beta_2$ is not an $n$-th root of unity. Then the vertices corresponding to $X_s$ and $Y_t$ are connected for all $1\leq s,t\leq n$. Hence $\mathcal{D}(V_1\oplus V_2)$ contains a $2n$-cycle, Corollary \ref{c4.5} forces $\GKdim(B(V_1\oplus V_2))=\infty$.

We now assume that $\beta_1\beta_2$ is an $n$-th root of unity. If $n=3$, then $|q|=3$, so there exists a number $0\leq k<3$ such that $\beta_1\beta_2=q^k$. In this case, vertices corresponding to $X_s$ and $Y_t$ are connected if and only if 
$$k+s+t-2\neq 0  \ (\mod 3).$$
Since $k\in \{0,1,2\}$, define $Y'_1=Y_{3-k}$, 
\[ Y'_{2}=
\begin{cases}
Y_{1}, & \ \mathrm{if}\  k=0, \\
Y_{4-k}, & \ \mathrm{if}\  k=1, 2 
\end{cases}\]
and 
\[ Y'_{3}=
\begin{cases}
Y_{2-k}, & \  \mathrm{if}\  k=0, 1,\\
Y_{3}, & \ \mathrm{if}\  k= 2.
\end{cases}\]
Clearly $\{Y'_{1}, Y'_{2}, Y'_{3}\}=\{Y_1,Y_2,Y_3\}$ is a basis of $V_2$.
 Let $a_1, a_2,a_3$ be the vertices of $\mathcal{D}(V_1\oplus V_2)$ corresponding to $X_1,X_2,X_3$, and $b_1,b_2,b_3$ those corresponding to $Y'_{1}, Y'_{2}, Y'_{3}$. A direct verification shows that $a_1$ is connected to $b_1$ and $b_3$, $a_2$ to $b_3$ and $b_2$, and $a_3$ to $b_2$ and $b_1$. Therefore $\mathcal{D}(V_1\oplus V_2)$ contains a $6$-cycle, whence $\GKdim(B(V_1\oplus V_2))=\infty$ by Corollary \ref{c4.5}.

Finally, suppose $n>3$. Let $k\in \{0,1,\ldots, n-1\}$ be the integer satisfying $\beta_1\beta_2=q^k$. Then the vertices corresponding to $X_s$ and $Y_t$ are connected exactly when
$$k+s+t-2\neq 0 \ (\mod n).$$ 
Define
\[ Y''_{1}=
\begin{cases}
Y_{n-k+2}, & \ \mathrm{if}\  2< k <  n ,\\
Y_{2-k}, & \ \mathrm{if}\  0\leq k\leq 2,
\end{cases}\]
and 
\[ Y''_{2}=
\begin{cases}
Y_{n-k+3}, & \ \mathrm{if}\ 3< k < n ,\\
Y_{3-k}, & \ \mathrm{if} \ 0\leq k\leq 3.
\end{cases}\]
Let $a_1,a_2$ be the vertices of $\mathcal{D}(V_1\oplus V_2)$ corresponding to $X_1,X_2$, and $b_1,b_2$ those corresponding to $Y''_{1}, Y''_{2}$. One easily check that both $a_1$ and $a_2$ are connected to $b_1$ and $b_2$. Thus
$\mathcal{D}(V_1\oplus V_2)$ contains a $4$-cycle, we also obtain $\GKdim(B(V_1\oplus V_2))=\infty$ by Corollary \ref{c4.5}.
\end{proof}

Compared with the families $\mathcal{F}_n$ for $n\geq 3$, the study of GK-dimension for Nichols algebras generated by objects in $\mathcal{F}_2$ poses greater challenges.  This problem will be addressed in the next two sections.

\section{Nichols algebras of objects in $\mathcal{F}_2^{(T2)}$}

Throughout this section, let $G$ be a finite abelian group that can be generated by $3$-elements, endowed with a nonabelian $3$-cocycle $\Phi$, and assume that $\mathcal{F}_2^{(T2)}$ is nonempty.

\subsection{A standard form}
In this subsection, we take an object $V\in \mathcal{F}_2^{(T2)}$ in standard form (see Subsection 2.3 for the definition). Write $V=V_1\oplus V_2\oplus V_3$ as a direct sum of simple objects, and let $g_i=g_{V_i}$ for $1\leq i\leq 3$. Thus $V_1, V_2$, $V_3$ are all of type (T2), and  because $V$ is in standard form, we have  $$G=\langle g_1\rangle \times \langle g_2\rangle \times \langle g_3\rangle.$$
Set $m_i=|g_i|$ for $1\leq i\leq 3$. The main result of this subsection is the following.
\begin{proposition}\label{p6.1}
Let $\Phi$ be a $3$-cocycle on $G$ of the form 
\begin{equation}\label{eq6.1}
\Phi(g_1^{i_1}g_2^{i_2}g_3^{i_3}, g_1^{j_1}g_2^{j_2}g_3^{j_3}, g_1^{k_1}g_2^{k_2}g_3^{k_3})=\zeta_{(m_1,m_2,m_3)}^{ci_1j_2k_3}, 
\end{equation}
 where $1\leq c< (m_1,m_2,m_3)$,  $0\leq i_1,j_1,k_1<m_1, 0\leq i_2,j_2,k_2<m_2$, and $0\leq i_3,j_3,k_3<m_3$. Then we have $\GKdim(B(V))=\infty.$
\end{proposition}

We shall prove the proposition through a sequence of lemmas. Throughout the remainder of this subsection, the $3$-cocycle $\Phi$ is assumed to be of the form \eqref{eq6.1}.  

According to Proposition \ref{p5.6}, the object $V_1$ has a basis $\{X_1,X_2\}$ satisfying 
\begin{eqnarray}
&&g_1\triangleright X_i=-X_i, \ i=1,2,\label{eq6.2}\\
&&g_2\triangleright X_1=\beta_1 X_1,  g_2\triangleright X_2=-\beta_1 X_2,\label{eq6.3}\\
&&g_3\triangleright X_1= X_2,  g_3\triangleright X_2=\gamma_1 X_1,\label{eq6.4}
\end{eqnarray}
with constants $\beta_1,\gamma_1\in \k^*$ subject to $\beta_1^{m_2}=1, \ \gamma_1^{\frac{m_3}{2}}=1$. Here \eqref{eq6.2} follows from the fact that $V_1$ is of type (T2).

Similarly, $V_2$ has a basis $\{Y_1,Y_2\}$ for which
\begin{eqnarray}
&&g_2\triangleright Y_i=-Y_i, \ i=1,2,\label{eq6.5}\\
&&g_1\triangleright Y_1=\beta_2 Y_1,  g_1\triangleright Y_2=-\beta_2 Y_2,\label{eq6.6}\\
&&g_3\triangleright Y_1= Y_2,  g_3\triangleright Y_2=\gamma_2 Y_1,\label{eq6.7}
\end{eqnarray}
with $\beta_2,\gamma_2\in \k^*$ satisfying $\beta_2^{m_1}=1, \ \gamma_2^{\frac{m_3}{2}}=1$.  

Finally, $V_3$ has a basis $\{Z_1,Z_2\}$ such that
\begin{eqnarray}
&&g_3\triangleright Z_i=-Z_i, \ i=1,2,\label{eq6.8} \\
&&g_2\triangleright Z_1=\beta_3 Z_1,  g_2\triangleright Z_2=-\beta_3 Z_2,\label{eq6.9}\\
&&g_1\triangleright Z_1= Z_2,  g_1\triangleright Z_2=\gamma_3 Z_1,\label{eq6.10}
\end{eqnarray}
where $\beta_3,\gamma_3\in \k^*$ are constants with $\beta_3^{m_2}=1, \gamma_3^{\frac{m_1}{2}}=1.$

\begin{lemma}\label{l6.2}
If $\beta_1\beta_2\neq \pm 1$, then $\GKdim(B(V_1\oplus V_2))=\infty.$
\end{lemma}
\begin{proof}
Observe that $\{X_1,X_2,Y_1,Y_2\}$ forms a diagonal basis for $V_1\oplus V_2$. When $\beta_1\beta_2\neq \pm 1$, the generalized Dynkin diagram $\mathcal{D}(V_1\oplus V_2)$ is exactly the $4$-cycle:
\[ {\setlength{\unitlength}{1.5mm}
\Echainfour{$\beta_1\beta_2$}{$-1$}{$-1$}{$-1$}{$-1$}{$-\beta_1\beta_2$}{$\beta_1\beta_2$}{$-\beta_1\beta_2$}} \quad .\] By Corollary \ref{c4.5}, it follows that $\GKdim(B(V_1\oplus V_2))=\infty$.
\end{proof}

\begin{lemma}\label{l6.3}
If $\beta_3\sqrt{\gamma_2}\neq \pm 1$, then $\GKdim(B(V_2\oplus V_3))=\infty.$
\end{lemma}
\begin{proof}
Define $Y'_1=\sqrt{\gamma_2}Y_1+Y_2$ and $Y'_2=\sqrt{\gamma_2}Y_1-Y_2$. Then we have 
\begin{equation}\label{eq6.11}
\begin{split}
g_3\triangleright Y'_1&= g_3\triangleright (\sqrt{\gamma_2}Y_1+Y_2)\\
&= \sqrt{\gamma_2}Y_2+\gamma_2Y_1\\
&= \sqrt{\gamma_2}(\sqrt{\gamma_2}Y_1+Y_2)\\
&= \sqrt{\gamma_2} Y'_1,
\end{split}
\end{equation}
and similarly
\begin{equation}\label{eq6.12}
\begin{split}
g_3\triangleright Y'_2&= g_3\triangleright (\sqrt{\gamma_2}Y_1-Y_2)\\
&= \sqrt{\gamma_2}Y_2-\gamma_2Y_1\\
&= -\sqrt{\gamma_2}(\sqrt{\gamma_2}Y_1-Y_2)\\
&=- \sqrt{\gamma_2} Y'_2.
\end{split}
\end{equation}
Clearly $\{Y'_1,Y'_2,Z_1,Z_2\}$ is a diagonal basis of $V_2\oplus V_3$. If $\beta_3\sqrt{\gamma_2}\neq \pm1$, then the generalized Dynkin diagram $\mathcal{D}(V_2\oplus V_3)$ coincides with
\[ {\setlength{\unitlength}{1.5mm}
\Echainfour{$\beta_3\sqrt{\gamma_2}$}{$-1$}{$-1$}{$-1$}{$-1$}{$-\beta_3\sqrt{\gamma_2}$}{$\beta_3\sqrt{\gamma_2}$}{$-\beta_3\sqrt{\gamma_2}$}} \quad .\] 
Hence Corollary \ref{c4.5} implies $\GKdim(B(V_2\oplus V_3))=\infty$.
\end{proof}

\begin{lemma}\label{l6.4}
If $\gamma_1\gamma_3\neq 1$, then $\GKdim(B(V_1\oplus V_3))=\infty.$
\end{lemma}
\begin{proof}
Define $X'_1=\sqrt{\gamma_1}X_1+X_2, X'_2=\sqrt{\gamma_1}X_1-X_2$, and $Z'_1=\sqrt{\gamma_3}Z_1+Z_2, Z'_2=\sqrt{\gamma_3}Z_1-Z_2$. As in the computations of \eqref{eq6.11} and \eqref{eq6.12}, one obtains 
\begin{equation*}
\begin{split}
g_3\triangleright X'_1=\sqrt{\gamma_1}X'_1, & \ \ \  g_3\triangleright X'_2=-\sqrt{\gamma_1}X'_2,\\
g_1\triangleright Z'_1=\sqrt{\gamma_3}Z'_1, & \ \ \  g_1\triangleright Z'_2=-\sqrt{\gamma_3}Z'_2.
\end{split}
\end{equation*}
Thus $\{X'_1,X'_2,Z'_1,Z'_2\}$ is a diagonal basis of $V_1\oplus V_3$. If $\gamma_1\gamma_3\neq 1$, the $\mathcal{D}(V_1\oplus V_3)$ becomes the $4$-cycle
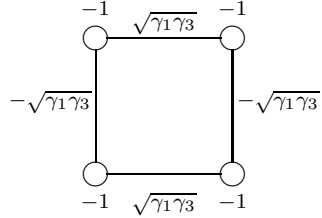
\[ {\setlength{\unitlength}{1.5mm}
\Echainfour{$\sqrt{\gamma_1\gamma_3}$}{$-1$}{$-1$}{$-1$}{$-1$}{$-\sqrt{\gamma_1\gamma_3}$}{$\sqrt{\gamma_1\gamma_3}$}{$-\sqrt{\gamma_1\gamma_3}$}} \quad .\] 
This implies that $\GKdim(B(V_1\oplus V_3))=\infty$ by Corollary \ref{c4.5}.
\end{proof}

\begin{lemma}\label{l6.5}
Suppose that
\begin{equation}\label{eq6.13}
\beta_1\beta_2=1,\  \beta_3\sqrt{\gamma_2}=\pm 1,\  {\gamma_1\gamma_3}= 1,
\end{equation}
 then $\GKdim(B(V))=\infty.$
\end{lemma}

\begin{proof}
The proof will be organized in four steps. We note that $B(V)$ is $\N_0^3$-graded with $\deg(V_i)=e_i$ for $1\leq i\leq 3$, where $\{e_1,e_2,e_3\}$ is a standard basis of $\N_0^3$. 

{\bf Step 1}. We prove that  
\begin{equation}\label{eq6.14}
\Phi(g_1^{i_1}g_2^{i_2}g_3^{i_3}, g_1^{j_1}g_2^{j_2}g_3^{j_3}, g_1^{k_1}g_2^{k_2}g_3^{k_3})=(-1)^{i_1j_2k_3}.
\end{equation}
From the explicit formula \eqref{eq6.1}, a direct calculation yields $$\frac{{\Phi}_{g_1}(g_2,g_3)}{{\Phi}_{g_1}(g_3,g_2)}=\frac{\Phi(g_1,g_2,g_3)\Phi(g_2,g_3,g_1)\Phi(g_3,g_1,g_2)}{\Phi(g_3,g_1,g_2)\Phi(g_1,g_3,g_2)\Phi(g_3,g_2,g_1)}=\Phi(g_1,g_2,g_3).$$
On the other hand, because $\dim(V_1)=|\frac{{\Phi}_{g_1}(g_2,g_3)}{{\Phi}_{g_1}(g_3,g_2)}|=2$, we have $|\Phi(g_1,g_2,g_3)|=2$.
Hence $$\Phi(g_1,g_2,g_3)=\zeta_{(m_1,m_2,m_3)}^{c}=-1.$$ Therefore the $3$-cocycle $\Phi$ necessarily takes the form \eqref{eq6.14}.

{\bf Step 2}. We show that the subspaces $\ad_{V_1}(V_2), \ad_{V_1}(V_3)$ and $\ad_{V_2}(V_3)$ of $B(V)$ are all $2$-dimensional.

First, we compute the coproduct of elements in $\ad_{V_2}(V_3)$. Using \eqref{eq3.22} we obtain
\begin{equation}\label{eq6.15}
\begin{split}
\D(\ad_{Y_1}(Z_1))=&\ad_{Y_1}(Z_1)\otimes 1+1\otimes \ad_{Y_1}(Z_1)+Y_1\otimes Z_1-g_3\triangleright Y_1\otimes g_2 \triangleright Z_1\\
=&1\otimes \ad_{Y_1}(Z_1)+ \ad_{Y_1}(Z_1)\otimes 1+(Y_1-\beta_3Y_2)\otimes Z_1.
\end{split}
\end{equation}
Similarly,  
\begin{eqnarray}
&&\D(\ad_{Y_1}(Z_2))=1\otimes \ad_{Y_1}(Z_2)+ \ad_{Y_1}(Z_2)\otimes 1+(Y_1+\beta_3 Y_2)\otimes Z_2, \label{eq6.16}\\
&&\D(\ad_{Y_2}(Z_1))=1\otimes \ad_{Y_2}(Z_1)+ \ad_{Y_2}(Z_1)\otimes 1+(Y_2-\beta_3\gamma_2 Y_1)\otimes Z_1,\label{eq6.17}\\
&&\D(\ad_{Y_2}(Z_2))=1\otimes \ad_{Y_2}(Z_2)+ \ad_{Y_2}(Z_2)\otimes 1 +(Y_2+\beta_3\gamma_2 Y_1)\otimes Z_2. \label{eq6.18}
\end{eqnarray}
From $\beta_3^2\gamma_2=1$ (by \eqref{eq6.13}) and equations \eqref{eq6.15} and \eqref{eq6.17}, we obtain 
\begin{equation}\label{eq6.19}
\begin{split}
&\D(\ad_{Y_1}(Z_1)+\beta_3\ad_{Y_2}(Z_1))\\
&=(\ad_{Y_1}(Z_1)+\beta_3\ad_{Y_2}(Z_1))\otimes 1+1\otimes (\ad_{Y_1}(Z_1)+\beta_3\ad_{Y_2}(Z_1)).
\end{split}
\end{equation}
If $\ad_{Y_1}(Z_1)+\beta_3\ad_{Y_2}(Z_1)\neq 0$, it would be a nonzero primitive element of length $2$, which contradicts the definition of a Nichols algebra. Hence
\begin{equation}
\ad_{Y_1}(Z_1)+\beta_3\ad_{Y_2}(Z_1)=0. \label{eq6.21}
\end{equation}
From \eqref{eq6.16} and \eqref{eq6.18}, we similarly obtain 
\begin{equation}\label{eq6.20}
\begin{split}
&\D(\ad_{Y_1}(Z_2)-\beta_3\ad_{Y_2}(Z_2))\\
&=(\ad_{Y_1}(Z_2)-\beta_3\ad_{Y_2}(Z_2))\otimes 1+1\otimes (\ad_{Y_1}(Z_2)-\beta_3\ad_{Y_2}(Z_2)),
\end{split}
\end{equation}
and consequently 
\begin{eqnarray}
\ad_{Y_1}(Z_2)-\beta_3\ad_{Y_2}(Z_2)=0. \label{eq6.22}
\end{eqnarray} 
Therefore, by \eqref{eq6.21} and \eqref{eq6.22}, the set $\{\ad_{Y_1}(Z_1), \ad_{Y_1}(Z_2)\}$ forms a basis of $\ad_{V_2}(V_3)$.

Next we consider $\ad_{V_1}(V_2)$.  Analogous to \eqref{eq6.15}-\eqref{eq6.18} we obtain
\begin{eqnarray}
&&\D(\ad_{X_1}(Y_1))=1\otimes \ad_{X_1}(Y_1) +\ad_{X_1}(Y_1)\otimes 1,\label{eq6.23} \\
&&\D(\ad_{X_1}(Y_2))=1\otimes \ad_{X_1}(Y_2)+ \ad_{X_1}(Y_2)\otimes 1+2X_1\otimes Y_2,\label{eq6.24}\\
&&\D(\ad_{X_2}(Y_1))=1\otimes \ad_{X_2}(Y_1)+\ad_{X_2}(Y_1)\otimes 1+2X_2 \otimes Y_1,\label{eq6.25}\\
&&\D(\ad_{X_2}(Y_2))=1\otimes \ad_{X_2}(Y_2)+\ad_{X_2}(Y_2)\otimes 1.\label{eq6.26}
\end{eqnarray}
Equations \eqref{eq6.23} and \eqref{eq6.26} force
\begin{equation}\label{eq6.27}
\ad_{X_1}(Y_1)=0,\  \ad_{X_2}(Y_2)=0.
\end{equation} 
Consequently $\{\ad_{X_1}(Y_2), \ad_{X_2}(Y_1)\}$ is a basis of $\ad_{V_1}(V_2)$.

Finally, we determine a basis of $\ad_{V_1}(V_3)$. Using \eqref{eq3.22} we have 
\begin{eqnarray}
&&\D(\ad_{X_1}(Z_1))=1\otimes \ad_{X_1}(Z_1) +\ad_{X_1}(Z_1)\otimes 1+X_1\otimes Z_1-X_2\otimes Z_2,\label{eq6.28} \\
&&\D(\ad_{X_1}(Z_2))=1\otimes \ad_{X_1}(Z_2)+ \ad_{X_1}(Z_2)\otimes 1+X_1\otimes Z_2-\gamma_3 X_2\otimes Z_1,\label{eq6.29} \\
&&\D(\ad_{X_2}(Z_1))=1\otimes \ad_{X_2}(Z_1)+\ad_{X_2}(Z_1)\otimes 1+ X_2\otimes Z_1-\gamma_1X_1\otimes Z_2,\label{eq6.30}\\
&&\D(\ad_{X_2}(Z_2))=1\otimes \ad_{X_2}(Z_2)+\ad_{X_2}(Z_2)\otimes 1 -X_1\otimes Z_1+X_2\otimes Z_2.\label{eq6.31}
\end{eqnarray}
Because $\gamma_1\gamma_3=1$ by \eqref{eq6.13}, the last four identities imply that
\begin{eqnarray}
&&\ad_{X_1}(Z_1)+\ad_{X_2}(Z_2)=0, \label{eq6.32}\\
&&\ad_{X_1}(Z_2)+\gamma_3 \ad_{X_2}(Z_1)=0.\label{eq6.33}
\end{eqnarray}
Hence $\{\ad_{X_1}(Z_1), \ad_{X_1}(Z_2)\}$ is a basis of $\ad_{V_1}(V_3)$.

{\bf Step 3.} Set $W=\ad_{V_1}(\ad_{V_2}(V_3))+\ad_{V_2}(\ad_{V_1}(V_3))$. We will show that the four elements $$\ad_{X_1}(\ad_{Y_1}(Z_1)), \ad_{X_2}(\ad_{Y_1}(Z_2)),\ad_{Y_1}(\ad_{X_1}(Z_2)), \ad_{Y_2}(\ad_{X_1}(Z_2))$$ are nonzero and linearly independent in $W$, hence $W$ is a nonzero subobject of $B(V)$.  

First we compute their coproducts. Using \eqref{eq3.24} we have 
\begin{equation}\label{eq6.34}
\begin{split}
&\D(\ad_{X_1}(\ad_{Y_1}(Z_1)))\\
&=1\otimes \ad_{X_1}(\ad_{Y_1}(Z_1))+\ad_{X_1}(\ad_{Y_1}(Z_1))\otimes 1\\
& \ \ + \big[X_1\otimes \ad_{Y_1}(Z_1)-(g_2g_3)\triangleright X_1\otimes g_1\triangleright \ad_{Y_1}(Z_1)\big]\\
&\ \ +\big[\Phi(g_1,g_2,g_3)\Phi^{-1}(g_2,g_1,g_3)g_1\triangleright Y_1\otimes \ad_{X_1} Z_1 \\
& \ \ \ \ \ \ \ \ \ \ \ \ \ -\Phi(g_1,g_3,g_2)g_1g_3\triangleright Y_1\otimes \ad_{X_1}(g_2 \triangleright Z_1) \big]\\
&\ \ +\Phi(g_1,g_2,g_3)\big[ X_1Y_1\otimes Z_1-X_1g_k\triangleright Y_1\otimes g_2\triangleright Z_1- g_1\triangleright Y_1 g_3\triangleright X_1\otimes g_1\triangleright Z_1\\
&\ \ \ \ \ \ \ \ \ \ \ \ \ \ \ \ \ \ \ \ \ \ +\Phi(g_2,g_1,g_3)\Phi(g_3,g_1,g_2)g_1g_3\triangleright Y_1g_3\triangleright X_1\otimes g_1g_2\triangleright Z_1\big]\\
&=1\otimes \ad_{X_1}(\ad_{Y_1}(Z_1)) +\ad_{X_1}(\ad_{Y_1}(Z_1))\otimes 1+ X_1\otimes \ad_{Y_1}(Z_1)\\
&\ \ \ +X_2\otimes  \ad_{Y_1}(Z_2)-\beta_2(Y_1+\beta_3Y_2)\otimes \ad_{X_1}(Z_1)\\
&\ \ \ +(\beta_3X_1Y_2-X_1Y_1)\otimes Z_1+ \beta_2(Y_1X_2+\beta_3Y_2X_2)\otimes Z_2.
\end{split}
\end{equation}
Here the second identity follows from \eqref{eq6.14} and \eqref{eq6.2}-\eqref{eq6.10}.
Clearly $\D(\ad_{X_1}(\ad_{Y_1}(Z_1)))\neq 0$, hence $\ad_{X_1}(\ad_{Y_1}(Z_1))\neq 0$.

In the same way we obtain

\begin{equation}\label{eq6.35}
\begin{split}
&\D( \ad_{X_2}(\ad_{Y_1}(Z_2)))\\
&=1\otimes  \ad_{X_2}(\ad_{Y_1}(Z_2)) + \ad_{X_2}(\ad_{Y_1}(Z_2))\otimes 1-\beta_1\gamma_1X_1\otimes \ad_{Y_1}(Z_1)\\
&\ \ \ +X_2\otimes  \ad_{Y_1}(Z_2)+\beta_2(Y_1-\beta_3Y_2)\otimes \ad_{X_1}(Z_1)\\
&\ \ \ +\beta_2(Y_1X_1-\beta_3Y_2X_1)\otimes Z_1-(X_2Y_1+\beta_3X_2Y_2)\otimes Z_2,
\end{split}
\end{equation}

\begin{equation}\label{eq6.36}
\begin{split}
&\D(\ad_{Y_1}(\ad_{X_1}(Z_2)))\\
&=1\otimes \ad_{Y_1}(\ad_{X_1}(Z_2))+\ad_{Y_1}(\ad_{X_1}(Z_2))\otimes 1-\beta_1X_1\otimes \ad_{Y_1}(Z_2)\\
&-\beta_1\gamma_3 X_2\otimes \ad_{Y_1}(Z_1)+(Y_1-\beta_3Y_2)\otimes \ad_{X_1}(Z_2)\\
&-\gamma_3(Y_1X_2+\beta_1\beta_3X_2Y_2)\otimes Z_1+(Y_1X_1+\beta_1\beta_3X_1Y_2)\otimes Z_2,
\end{split}
\end{equation}
and 
\begin{equation}\label{eq6.37}
\begin{split}
&\D(\ad_{Y_2}(\ad_{X_1}(Z_2)))\\
&=1\otimes \ad_{Y_2}(\ad_{X_1}(Z_2))+\ad_{Y_2}(\ad_{X_1}(Z_2))\otimes 1-\beta_1X_1\otimes \ad_{Y_2}(Z_2)\\
&-\beta_1\gamma_3X_2\otimes \ad_{Y_2}(Z_1)+(Y_2+\beta_3\gamma_2Y_1)\otimes \ad_{X_1}(Z_2)\\
&-\gamma_3(Y_2X_2+\beta_1\beta_3\gamma_2X_2Y_1)\otimes Z_1+(Y_2X_1+\beta_1\beta_3\gamma_2X_1Y_1)\otimes Z_2.
\end{split}
\end{equation}
Because their coproducts are nonzero, the elements  $\ad_{X_2}(\ad_{Y_1}(Z_2)),$ $\ad_{Y_1}(\ad_{X_1}(Z_2)),$ $ \ad_{Y_2}(\ad_{X_1}(Z_2))$ are also nonzero. 

We now prove that these four elements are linearly independent. Set 
$$T=c_1 \ad_{X_1}(\ad_{Y_1}(Z_1))+c_2\ad_{X_2}(\ad_{Y_1}(Z_2))+c_3 \ad_{Y_1}(\ad_{X_1}(Z_2))+c_4 \ad_{Y_2}(\ad_{X_1}(Z_2))$$
with $c_1,c_2,c_3,c_4\in \k$ and assume $T=0$. We must show that all coefficients are zero.
Using \eqref{eq6.34}-\eqref{eq6.37}, the component of $\D(T)$ of $\N_0^3$-degree $(e_2, e_1+e_3)$ equals
\begin{equation}\label{eq6.38}
\begin{split}
&\big[(c_2-c_1)\beta_2Y_1-(c_1+c_2)\beta_2\beta_3Y_2\big]\otimes \ad_{X_1}(Z_1)\\
&+\big[(c_3+c_4\beta_3\gamma_2)Y_1+(c_4-c_3\beta_3)Y_2\big]\otimes \ad_{X_1}(Z_2)\\
=&\big[(c_2-c_1)\beta_2Y_1-(c_1+c_2)\beta_2\beta_3Y_2\big]\otimes \ad_{X_1}(Z_1)\\
&+\big[(c_3+\beta_3\gamma_2c_4)Y_1+\beta_3(\beta_3\gamma_2c_4-c_3)Y_2\big]\otimes \ad_{X_1}(Z_2),
\end{split}
\end{equation}
where the simplification uses \eqref{eq6.13}. Since $T=0$, this expression must vanish. Because $\{Y_1,Y_2\}$ is a basis of $V_2$, and $\{\ad_{X_1}(Z_1), \ad_{X_1}(Z_2)\}$ is a basis of $\ad_{V_1}(V_3)$, we obtain
\begin{equation}\label{6.39}
\begin{split}
&c_2-c_1=0, \ c_1+c_2=0,\\
&c_3+\beta_3\gamma_2c_4=0,\ c_4\beta_3\gamma_2-c_3=0. 
\end{split}
\end{equation}
It follows immediately that $c_1=c_2=c_3=c_4=0$. Hence the four elements $\ad_{X_1}(\ad_{Y_1}(Z_1))$, $\ad_{X_2}(\ad_{Y_1}(Z_2))$, $\ad_{Y_1}(\ad_{X_1}(Z_2))$ and $\ad_{Y_2}(\ad_{X_1}(Z_2))$ are linearly independent.

{\bf Step 4}. We show that $\GKdim(A(W))=\infty$, where $A(W)$ is the subalgebra of $B(V)$ generated by $W$. Because $A(W)$ is an $N_0$-graded subalgebra of $B(V)$, this will imply $\GKdim(B(V))=\infty$.

By Proposition \ref{p3.11}, $A(W)$ is a pre-Nichols algebra of $W$. Hence it suffices to prove that $\GKdim(B(W))=\infty$, because there is a canonical projection from $A(W)$ onto $B(W)$. Since $G_W=\langle g_1g_2g_3\rangle$ is a cyclic group, $B(W)$ is of diagonal type. Therefore it is enough to verify that the generalized Dynkin diagram $\mathcal{D}(W)$ is not of finite type.

Direct calculation gives
\begin{equation}\label{eq6.39}
\begin{split}
&g_1g_2g_3 \triangleright \ad_{X_1}(\ad_{Y_1}(Z_1))\\
&=\Phi_{g_1g_2g_3}(g_1g_2,g_3)\Phi_{g_1g_2g_3}(g_1,g_2)g_1 \triangleright \{g_2\triangleright [g_3 \triangleright \ad_{X_1}(\ad_{Y_1}(Z_1))]\}\\
&=-g_1 \triangleright \{g_2\triangleright [g_3 \triangleright \ad_{X_1}(\ad_{Y_1}(Z_1))]\}\\
&=-g_1 \triangleright \{g_2\triangleright \ad_{X_2}(\ad_{Y_2}(Z_1))\}\\
&=\beta_1\beta_3 g_1 \triangleright \ad_{X_2}(\ad_{Y_2}(Z_1))\\
&=-\beta_1\beta_2\beta_3 \ad_{X_2}(\ad_{Y_2}(Z_2))\\
&=- \ad_{X_2}(\ad_{Y_1}(Z_2)).
\end{split}
\end{equation}
Here the sixth identity follows from \eqref{eq6.13} and \eqref{eq6.22}.
Furthermore,
\begin{equation}\label{eq6.40}
\begin{split}
&g_1g_2g_3 \triangleright \ad_{X_2}(\ad_{Y_1}(Z_2))\\
&=\Phi_{g_1g_2g_3}(g_1g_2,g_3)\Phi_{g_1g_2g_3}(g_1,g_2)g_1 \triangleright \{g_2\triangleright [g_3 \triangleright \ad_{X_2}(\ad_{Y_1}(Z_2))]\}\\
&=-g_1 \triangleright \{g_2\triangleright [g_3 \triangleright \ad_{X_2}(\ad_{Y_1}(Z_2))]\}\\
&=-\gamma_1 g_1 \triangleright \{g_2\triangleright \ad_{X_1}(\ad_{Y_2}(Z_2))\}\\
&=\beta_1\beta_3\gamma_1 g_1 \triangleright  \ad_{X_1}(\ad_{Y_2}(Z_2))\\
&=-\beta_1\beta_2\beta_3\gamma_1\gamma_3  \ad_{X_1}(\ad_{Y_2}(Z_1))\\
&= \ad_{X_1}(\ad_{Y_1}(Z_1)).
\end{split}
\end{equation}
Here the last step uses \eqref{eq6.21}.
Similarly, 
\begin{equation}\label{eq6.41}
\begin{split}
&g_1g_2g_3 \triangleright \ad_{Y_1}(\ad_{X_1}(Z_2))\\
&=\Phi_{g_1g_2g_3}(g_1g_2,g_3)\Phi_{g_1g_2g_3}(g_1,g_2)g_1 \triangleright \{g_2\triangleright [g_3 \triangleright \ad_{Y_1}(\ad_{X_1}(Z_2))]\}\\
&=-g_1 \triangleright \{g_2\triangleright [g_3 \triangleright \ad_{Y_1}(\ad_{X_1}(Z_2))]\}\\
&=g_1 \triangleright \{g_2\triangleright \ad_{Y_2}(\ad_{X_2}(Z_2))\}\\
&=\beta_1\beta_3 g_1 \triangleright \ad_{Y_2}(\ad_{X_2}(Z_2))\\
&=\beta_1\beta_2\beta_3\gamma_3  \ad_{Y_2}(\ad_{X_2}(Z_1))\\
&=-\beta_3 \ad_{Y_2}(\ad_{X_1}(Z_2)),
\end{split}
\end{equation}
and 
\begin{equation}\label{eq6.42}
\begin{split}
&g_1g_2g_3 \triangleright \ad_{Y_2}(\ad_{X_1}(Z_2))\\
&=\Phi_{g_1g_2g_3}(g_1g_2,g_3)\Phi_{g_1g_2g_3}(g_1,g_2)g_1 \triangleright \{g_2\triangleright [g_3 \triangleright \ad_{Y_2}(\ad_{X_1}(Z_2))]\}\\
&=-g_1 \triangleright \{g_2\triangleright [g_3 \triangleright \ad_{Y_2}(\ad_{X_1}(Z_2))]\}\\
&=\gamma_2 g_1 \triangleright \{g_2\triangleright \ad_{Y_1}(\ad_{X_2}(Z_2))\}\\
&=\beta_1\beta_3\gamma_2 g_1\triangleright \ad_{Y_1}(\ad_{X_2}(Z_2))\\
&=-\beta_1\beta_2\beta_3\gamma_2\gamma_3 \ad_{Y_1}(\ad_{X_2}(Z_1))\\
&=-\beta_3\gamma_2\gamma_3 \ad_{Y_1}(\ad_{X_2}(Z_1))\\
&=\beta_3\gamma_2 \ad_{Y_1}(\ad_{X_1}(Z_2)).
\end{split}
\end{equation}
Now set $\mathrm{i}=\sqrt{-1}$ and define
\begin{eqnarray*}
E&=&\ad_{X_1}(\ad_{Y_1}(Z_1))+\mathrm{i}\ad_{X_2}(\ad_{Y_1}(Z_2)),\\
F&=&\ad_{Y_1}(\ad_{X_1}(Z_2))+\mathrm{i} \beta_3 \ad_{Y_2}(\ad_{X_1}(Z_2)).
\end{eqnarray*}
From \eqref{eq6.39}-\eqref{eq6.42} we obtain
\begin{equation*}
\begin{split}
g_1g_2g_3\triangleright E&=g_1g_2g_3\triangleright \big[\ad_{X_1}(\ad_{Y_1}(Z_1))+\mathrm{i}\ad_{X_2}(\ad_{Y_1}(Z_2))\big]\\
&=- \ad_{X_2}(\ad_{Y_1}(Z_2))+\mathrm{i}\ad_{X_1}(\ad_{Y_1}(Z_1))\\
&=\mathrm{i}\big[\ad_{X_1}(\ad_{Y_1}(Z_1))+\mathrm{i}\ad_{X_2}(\ad_{Y_1}(Z_2))\big]\\
&=\mathrm{i}E,
\end{split}
\end{equation*}
and 
\begin{equation*}
\begin{split}
g_1g_2g_3\triangleright F&=g_1g_2g_3\triangleright \big[\ad_{Y_1}(\ad_{X_1}(Z_2))+\mathrm{i} \beta_3 \ad_{Y_2}(\ad_{X_1}(Z_2))\big]\\
&=-\beta_3 \ad_{Y_2}(\ad_{X_1}(Z_2))+\mathrm{i}\beta_3 \beta_3\gamma_2 \ad_{Y_1}(\ad_{X_1}(Z_2))\\
&=-\beta_3 \ad_{Y_2}(\ad_{X_1}(Z_2))+\mathrm{i} \ad_{Y_1}(\ad_{X_1}(Z_2))\\
&=\mathrm{i} \big[\ad_{Y_1}(\ad_{X_1}(Z_2))+\mathrm{i} \beta_3 \ad_{Y_2}(\ad_{X_1}(Z_2))\big]\\
&=\mathrm{i}F.
\end{split}
\end{equation*}
These two equalities show that the generalized Dynkin diagram $\mathcal{D}(W)$ contains the subdiagram 
\[ {\setlength{\unitlength}{1.5mm}
\Dchaintwo{}{$\mathrm{i} $}{$-1$}{$\mathrm{i} $}} \quad .\]
Comparing with the set $\mathfrak{D}$ of generalized Dynkin diagrams of finite type, we see that $\mathcal{D}(W)$ is not in $\mathfrak{D}$. Hence Corollary \ref{c4.3} yields $\GKdim(B(W))=\infty$.

This completes the proof of the lemma.
\end{proof}

\begin{lemma}\label{l6.6}
Suppose that 
\begin{equation}\label{eq6.47}
\beta_1\beta_2=-1,  \beta_3\sqrt{\gamma_2}=\pm 1, {\gamma_1\gamma_3}= 1,
\end{equation}
 then $\GKdim(B(V))=\infty.$
\end{lemma}
\begin{proof}
Define a new basis $\{X'_1,X'_2\}$ of $V_1$ by $X'_1=X_2, X'_2=\gamma_1X_1$. Set $\beta'_1=-\beta_1$. Using \eqref{eq6.2}-\eqref{eq6.4} we obtain
\begin{eqnarray}
&&g_1\triangleright X'_i=-X'_i, \  \ i=1,2,\label{eq6.48}\\
&&g_2\triangleright X'_1=\beta'_1 X'_1, \ \  g_2\triangleright X'_2=-\beta'_1 X'_2,\label{eq6.49}\\
&&g_3\triangleright X'_1= X'_2, \ \  g_3\triangleright X'_2=\gamma_1 X'_1.\label{eq6.50}
\end{eqnarray}
Thus, with respect to the basis $\{X'_1,X'_2,Y_1,Y_2,Z_1,Z_2\}$ of $V$, the associated constants $\beta'_1, \gamma_1, \beta_2,\gamma_2, \beta_3$ and $\gamma_3$ satisfy
\begin{equation*}
\beta'_1\beta_2=1,  \beta_3\sqrt{\gamma_2}=\pm 1, {\gamma_1\gamma_3}= 1.
\end{equation*}
This implies $\GKdim(B(V))=\infty$ by Lemma \ref{l6.5}.
\end{proof}

Lemmas \ref{l6.2}-\ref{l6.6} complete the proof of Proposition \ref{p6.1}.

\subsection{General cases}
In this subsection, we determine $\GKdim(B(V))$ for an arbitrary minimal nondiagonal object $V\in \mathcal{F}_2^{(T2)}$. We first prove the following proposition.

\begin{proposition}\label{p6.7}
Let $V\in \mathcal{F}_2^{(T_2)}$ be an object in standard form. Then $\GKdim(B(V))=\infty.$
\end{proposition}
\begin{proof}
Let $V=V_1\oplus V_2\oplus V_3 \in \mathcal{F}_2^{(T_2)}$, and denote $g_i=g_{V_i}$ for $1\leq i\leq 3$. Since $V$ is of standard form, we have $G=\langle g_1\rangle \times \langle g_2\rangle \times g_3\rangle$.
Set $m_i=|g_i|$ for each $1\leq i\leq 3$. By \eqref{eq2.10}, up to $3$-cohomology, $\Phi$ can be expressed as 
\begin{equation}
\begin{split}
\Phi(g_1^{i_1}g_2^{i_2}g_3^{i_3},g_1^{j_1}g_2^{j_2}g_3^{j_3}, g_1^{k_1}g_2^{k_2}g_3^{k_3})&=\prod_{l=1}^{3}\zeta_{m_l}^{c_{l}i_{l}[\frac{j_{l}+k_{l}}{m_{l}}]}
\prod_{1\leq s<t\leq 3}\zeta_{m_{t}}^{c_{st}i_{t}[\frac{j_{s}+k_{s}}{m_{s}}]}   \times \zeta_{(m_{1},m_{2},m_{3})}^{ci_{1}j_{2}k_{3}},
\end{split}
\end{equation}
where $c_l, c_{st}$ and $c$ are integers satisfying $0\leq c_{l}< m_l$, $0\leq c_{st}< m_{t}$, $0\leq c< (m_{1},m_{2},m_{3})$.
Define the $3$-cocycle $\Psi$ on $G$ by 
$$\Psi(g_1^{i_1}g_2^{i_2}g_3^{i_3},g_1^{j_1}g_2^{j_2}g_3^{j_3}, g_1^{k_1}g_2^{k_2}g_3^{k_3})=\zeta_{(m_{1},m_{2},m_{3})}^{ci_{1}j_{2}k_{3}}.$$ 
Then 
\begin{equation}
\begin{split}
(\Phi\Psi^{-1})(g_1^{i_1}g_2^{i_2}g_3^{i_3},g_1^{j_1}g_2^{j_2}g_3^{j_3}, g_1^{k_1}g_2^{k_2}g_3^{k_3})=\prod_{l=1}^{3}\zeta_{m_l}^{c_{l}i_{l}[\frac{j_{l}+k_{l}}{m_{l}}]}
\prod_{1\leq s<t\leq 3}\zeta_{m_{t}}^{c_{st}i_{t}[\frac{j_{s}+k_{s}}{m_{s}}]}.
\end{split}
\end{equation}
By \eqref{eq2.11}, $\Phi\Psi^{-1}$ is an abelian 3-cocycle on $G$ . 

Let $\widehat{G}$ be the abelian group associated to $G$ defined in \eqref{eq2.12}, and let $\pi:\widehat{G} \to G$ be the projection given in \eqref{eq2.13}. By Lemma \ref{l3.4}, $B(V)$ is isomorphic to a Nichols algebra (still denoted by $B(V)$) in the category $^{\k \widehat{G}}_{\k \widehat{G}} \mathcal{YD}^{\pi^*{\Phi}}$. According to Proposition \ref{p2.3}, there exists a $2$-cochain $J$ on $\widehat{G}$ such that $\partial J=\pi^*(\Phi\Psi^{-1})$. Consequently, $${\pi^*{\Phi}\times \partial J^{-1}}={\pi^*{\Psi}},$$ and by Lemma \ref{l2.6}, $B(V^{J^{-1}})$ is a Nichols algebra of standard form in  $^{\k \widehat{G}}_{\k \widehat{G}} \mathcal{YD}^{\pi^*{\Psi}}$. Because $\pi^*{\Psi}$ is a $3$-cocycle on $\widehat{G}$ of the form \eqref{eq6.1}, Proposition \ref{p6.1} implies that 
$\GKdim(B(V^{J^{-1}}))=\infty$. Finally, applying Lemma \ref{l2.11} yields $\GKdim(B(V))=\infty$.
\end{proof}

To prove the main result of this section, we also require the following lemma.
\begin{lemma}\label{l6.9}
Let $H$ be a finite abelian group and let $P: H\to G$ be a group epimorphism. If $\Phi$ is a nonabelian $3$-cocycle on $G$, then the pull-back $P^*\Phi$ is also a nonabelian $3$-cocycle on $H$.
\end{lemma}
\begin{proof}
Assume, to the contrary, that $P^*\Phi$ is an abelian $3$-cocycle on $H$. Let $\widehat{H}$ be the finite abelian group associated to $H$ defined by \eqref{eq2.12}, and let $\pi: \widehat{H} \to H$ be the canonical epimorphism defined in \eqref{eq2.13}. By Proposition \ref{p2.3}, the pull-back $\pi^*( P^*\Phi)$ is a $3$-coboundary on $\widehat{H}$. On the other hand, the composition $P\circ \pi: \widehat{H}\to G$ is a group epimorphism, and $$(P\circ \pi)^*\Phi=\pi^*( P^*\Phi).$$ By Proposition \ref{p2.4}, the pull-back $(P\circ \pi)^*\Phi$ is not a $3$-coboundary on $\widehat{H}$. This contradiction completes the proof. 
\end{proof}

\begin{proposition}\label{p6.9}
For every $V\in \mathcal{F}_2^{(T2)}$, we have $\GKdim(B(V))=\infty.$
\end{proposition}
\begin{proof}
Let $V=V_1\oplus V_2\oplus V_3\in \mathcal{F}_2^{(T2)}$ be a direct sum of simple objects, and denote $g_i=g_{V_i}$ for $1\leq i\leq 3$. By the definition of $\mathcal{F}_2^{(T2)}$, we have $G=G_V=\langle g_1,g_2,g_3\rangle$. Define a finite abelian group  $H=\langle h_1\rangle \times \langle h_2\rangle\times \langle h_3\rangle$ with $|h_i|=|g_i|$ for $1\leq i\leq 3$. Then the map
\begin{eqnarray*}
P: H \To G, \ \ h^{i_1}_1h^{i_2}_2h^{i_3}_3 \to g^{i_1}_1g^{i_2}_2g^{i_3}_3
\end{eqnarray*}
is clearly a group epimorphism. Let $\iota:G\to H$ be a section of $P$ satisfying $\iota(g_i)=h_i, 1\leq i\leq 3$. By Lemma \ref{l3.4}, the Nichols algebra $B(V)$ also lives in $^{\k H}_{\k H} \mathcal{YD}^{P^*\Phi}$, with the $H$-degree of $V_i$ equal to $\iota(g_i)=h_i$ for each $i$. Because $H=\langle h_1\rangle \times \langle h_2\rangle\times \langle h_3\rangle$, the Nichols algebra $B(V)$ is in standard form in the category $^{\k H}_{\k H} \mathcal{YD}^{P^*\Phi}$.

On the other hand,  Lemma \ref{l3.3} guarantees that $V_1,V_2,V_3$ remain simple in $^{\k H}_{\k H} \mathcal{YD}^{P^*\Phi}$. Moreover, Lemma \ref{l6.9} tell us that the pull-back $P^*\Phi$ is a nonabelian $3$-cocycle on $H$. Hence, by Proposition \ref{p5.4}, $V$ is a minimal nondiagonal object in $^{\k H}_{\k H} \mathcal{YD}^{P^*\Phi}$. Finally, because each $V_i$ is of type (T2) for $1\leq i\leq 3$, we obtain $\GKdim(B(V))=\infty$ by Proposition \ref{p6.7}.
\end{proof}

\section{Nichols algebras of objects in $\mathcal{F}_2^{T}$}
So far we have proved that $\GKdim(B(V))=\infty$ for each $V\in \mathcal{F}_2^{(T2)}$. In this section, we shall prove that the same conclusion holds for all $V\in \mathcal{F}_2^{T}$. To ensure that $\mathcal{F}_2^{T}$ is non-empty, we continue to assume that $G$ can be generated by three elements and that $\Phi$ is a nonabelian $3$-cocycle on $G$.

Throughout this section, let $V=V_1\oplus V_2\oplus V_3\in \mathcal{F}_2^{T}$ be a fixed object, where each $V_i$ is a simple object of type (T1)-(T3). Write $g_i=g_{V_i}$ and $m_i=|g_i|$ for $1\leq i\leq 3$. By the definition of $\mathcal{F}_2^{T}$, we have $G=G_V=\langle g_1,g_2,g_3\rangle.$

\begin{proposition}\label{p7.1}
If $V\in \mathcal{F}_2^{(T1)}$, then $\GKdim{(B(V))}=\infty$.
\end{proposition}

\begin{proof}
Assume that $V\in  \mathcal{F}_2^{(T1)}$, thus its simple subobjects $V_1,V_2, V_3$ are all of type (T1). We shall show that $\GKdim(B(V_1\oplus V_2))=\infty$. Since $B(V_1\oplus V_2)$ is of diagonal type, it suffices to prove that the generalized Dynkin diagram $\mathcal{D}(V_1\oplus V_2)$ is not of finite type.

By Proposition \ref{p5.6}, $V_1$ has a basis $\{X_1,X_2\}$ such that
\begin{eqnarray*}
&&g_1\triangleright X_i= X_i,\  \ 1\leq i\leq 2,\\
&&g_2\triangleright X_1=\beta_1 X_1,\ g_2\triangleright X_2=-\beta_1 X_2, \\
&&g_3\triangleright X_1=X_{2}, \ g_3\triangleright X_2=\gamma_1 X_1,
\end{eqnarray*}
where $\beta_1, \gamma_1\in \k^*$ satisfy $\beta_1^{m_2}=\prod_{l=1}^{m_2-1}\Phi_{g_1}(g_2, g_2^l),$ 
$\gamma_1^{\frac{m_3}{2}}=\prod_{l=1}^{m_3-1}\Phi_{g_1}(g_3, g_3^l).$

Similarly, $V_2$ has a basis $\{Y_1,Y_2\}$ such that 
\begin{eqnarray*}
&&g_2\triangleright Y_i= Y_i,\  \ 1\leq i\leq 2,\\
&&g_1\triangleright Y_1=\beta_2 Y_1,\ g_1\triangleright Y_2=-\beta_2 Y_2, \\
&&g_3\triangleright Y_1=Y_{2}, \ g_3\triangleright Y_2=\gamma_2 Y_1,
\end{eqnarray*}
with $\beta_2, \gamma_2\in \k^*$ satisfying $\beta_2^{m_1}=\prod_{l=1}^{m_1-1}\Phi_{g_2}(g_1, g_1^l),$ 
$\gamma_2^{\frac{m_3}{2}}=\prod_{l=1}^{m_3-1}\Phi_{g_2}(g_3, g_3^l).$
Clearly, $\{X_1,X_2,Y_1,Y_2\}$ forms a diagonal basis of $V_1\oplus V_2$.

If $\beta_1\beta_2\neq \pm 1$, then $\mathcal{D}(V_1\oplus V_2)$ is the $4$-cycle:
\[ {\setlength{\unitlength}{1.5mm}
\Echainfour{$\beta_1\beta_2$}{$1$}{$1$}{$1$}{$1$}{$-\beta_1\beta_2$}{$\beta_1\beta_2$}{$-\beta_1\beta_2$}} \quad .\] 
In this case, we have $\GKdim(B(V_1\oplus V_2))=\infty$ by Corollary \ref{c4.5}. 

If $\beta_1\beta_2=\pm 1$, the diagram $\mathcal{D}(V_1\oplus V_2)$ becomes
\[ {\setlength{\unitlength}{1.5mm}
\Chainfour{$1$}{$-1$}{$1$}{$1$}{$-1$}{$1$}} \quad .\] 
Comparing with the collection $\mathfrak{D}$ of generalized Dynkin diagrams of finite type, we see that $\mathcal{D}(V_1\oplus V_2)$ does not belong to $\mathfrak{D}$. Hence, by Corollary \ref{c4.3}, we again conclude $\GKdim(B(V_1\oplus V_2))=\infty$.
\end{proof}

\begin{proposition}\label{p7.2}
If $V\in \mathcal{F}_2^{(T3)}$, then $\GKdim{(B(V))}=\infty$.
\end{proposition} 

\begin{proof}
Assume $V\in \mathcal{F}_2^{(T3)}$, thus $V_1,V_2,V_3$ are all of type (T3). By Proposition \ref{p5.6}, there exists a basis $\{X_1,X_2\}$ of $V_1$ such that 
\begin{eqnarray*}
&&g_1\triangleright X_i= \zeta_3 X_i,\  \ 1\leq i\leq 2;\\
&&g_2\triangleright X_1=\beta_1 X_1,\ g_2\triangleright X_2=-\beta_1 X_2; \\
&&g_3\triangleright X_1=X_{2}, \ g_3\triangleright X_2=\gamma_1 X_1,
\end{eqnarray*}
where $\zeta_3$ is a primitive third root of unity, and $\beta_1, \gamma_1\in \k^*$ satisfy $\beta_1^{m_2}=\prod_{l=1}^{m_2-1}\Phi_{g_1}(g_2, g_2^l),$ 
$\gamma_1^{\frac{m_3}{2}}=\prod_{l=1}^{m_3-1}\Phi_{g_1}(g_3, g_3^l).$

Similarly, $V_2$ has a basis $\{Y_1,Y_2\}$ such that 
\begin{eqnarray*}
&&g_2\triangleright Y_i= \zeta'_3Y_i,\  \ 1\leq i\leq 2;\\
&&g_1\triangleright Y_1=\beta_2 Y_1,\ g_1\triangleright Y_2=-\beta_2 Y_2; \\
&&g_3\triangleright Y_1=Y_{2}, \ g_3\triangleright Y_2=\gamma_2 Y_1.
\end{eqnarray*}
Here $\zeta'_3$ is also a primitive third root of unity, and $\beta_2, \gamma_2\in \k^*$ satisfy $\beta_2^{m_1}=\prod_{l=1}^{m_1-1}\Phi_{g_2}(g_1, g_1^l),$ 
$\gamma_2^{\frac{m_3}{2}}=\prod_{l=1}^{m_3-1}\Phi_{g_2}(g_3, g_3^l).$ It is clear that $\{X_1,X_2,Y_1,Y_2\}$ is a diagonal basis of $V_1\oplus V_2$.

If $\beta_1\beta_2\neq \pm 1$, the generalized Dynkin diagram $\mathcal{D}(V_1\oplus V_2)$ is 
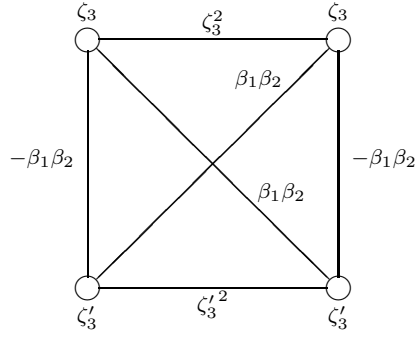
\[ {\setlength{\unitlength}{1.5mm}
\Nchainfour{$\zeta_3$}{$\zeta_3$}{$\zeta'_3$}{$\zeta'_3$}{$-\beta_1\beta_2$}{$\zeta_3^2$}{$-\beta_1\beta_2$}{$
{\zeta'_3}^2$}{$\beta_1\beta_2$}} \quad .\] 
If $\beta_1\beta_2= \pm 1$, the diagram $\mathcal{D}(V_1\oplus V_2)$ becomes
\[ {\setlength{\unitlength}{1.5mm}
\Echainfour{${\zeta'_3}^2$}{$\zeta'_3$}{$\zeta'_3$}{$\zeta_3$}{$\zeta_3$}{$-1$}{$\zeta_3^2$}{$-1$}} \quad .\] 
In both cases, $\mathcal{D}(V_1\oplus V_2)$ contains a $4$-cycle. By Corollary \ref{c4.5}, we get $\GKdim(B(V_1\oplus V_2))=\infty$.
\end{proof}
 
\begin{proposition}\label{p7.3}
If the simple objects $V_1,V_2$ and $V_3$ are not all of the same type among (T1)-(T3), then $\GKdim(B(V))=\infty.$
\end{proposition} 
\begin{proof}
Without loss of generality, we may assume that $V_1$ and $V_2$ are of distinct types.  
Then there are three possible cases: 
\begin{itemize}
\item[(1)] $V_1$ is of type (T1) and $V_2$ is of type (T2); \item[(2)] $V_1$ is of type (T1) and $V_2$ is of type (T3); \item[(3)] $V_1$ is of type (T2) and $V_2$ is of type (T3).  
\end{itemize}
We will examine each case separately.

Case (1). $V_1$ is of type (T1) and $V_2$ is of type (T2). 

By Proposition \ref{p5.6}, there is a basis $\{X_1,X_2\}$ of $V_1$ such that 
\begin{eqnarray}
&&g_1\triangleright X_i= X_i,\  \ 1\leq i\leq 2,\label{e7.1} \\
&&g_2\triangleright X_1=\beta_1 X_1,\ g_2\triangleright X_2=-\beta_1 X_2,\label{e7.2} \\
&&g_3\triangleright X_1=X_{2}, \ g_3\triangleright X_2=\gamma_1 X_1,\label{e7.3}
\end{eqnarray}
where $\beta_1, \gamma_1\in \k^*$ satisfy $\beta_1^{m_2}=\prod_{l=1}^{m_2-1}\Phi_{g_1}(g_2, g_2^l),$ $\gamma_1^{\frac{m_3}{2}}=\prod_{l=1}^{m_3-1}\Phi_{g_1}(g_3, g_3^l).$

Similarly,  $V_2$ has a basis $\{Y_1,Y_2\}$ such that 
\begin{eqnarray}
&&g_2\triangleright Y_i=-Y_i,\  \ 1\leq i\leq 2,\\
&&g_1\triangleright Y_1=\beta_2 Y_1,\ g_1\triangleright Y_2=-\beta_2 Y_2, \\
&&g_3\triangleright Y_1=Y_{2}, \ g_3\triangleright Y_2=\gamma_2 Y_1,
\end{eqnarray} 
where $\beta_2, \gamma_2\in \k^*$ satisfy $\beta_2^{m_1}=\prod_{l=1}^{m_1-1}\Phi_{g_2}(g_1, g_1^l),$ 
$\gamma_2^{\frac{m_3}{2}}=\prod_{l=1}^{m_3-1}\Phi_{g_2}(g_3, g_3^l).$ Hene we obtain a diagonal basis $\{X_1,X_2,Y_1,Y_2\}$ for $V_1\oplus V_2$.

If $\beta_1\beta_2\neq \pm 1$, the generalized Dynkin diagram $\mathcal{D}(V_1\oplus V_2)$ is the $4$-cycle:
\[ {\setlength{\unitlength}{1.5mm}
\Echainfour{$\beta_1\beta_2$}{$-1$}{$1$}{$1$}{$-1$}{$-\beta_1\beta_2$}{$\beta_1\beta_2$}{$-\beta_1\beta_2$}} \quad .\] 
Corollary \ref{c4.5} then forces $\GKdim(B(V_1\oplus V_2))=\infty$.

If $\beta_1\beta_2=\pm 1$, the diagram $\mathcal{D}(V_1\oplus V_2)$ becomes
\[ {\setlength{\unitlength}{1.5mm}
\Chainfour{$1$}{$-1$}{$-1$}{$1$}{$-1$}{$-1$}} \quad .\] 
Comparing with the set $\mathfrak{D}$ of generalized Dynkin diagrams of finite type, we see that $\mathcal{D}(V_1\oplus V_2)$ is not in $\mathfrak{D}$. Hence, by Corollary \ref{c4.3}, we again obtain $\GKdim(B(V_1\oplus V_2))=\infty$.

(2). $V_1$ is of type (T1) and $V_2$ is of type (T3). 

Take the same basis $\{X_1, X_2\}$ of $V_1$ as in \eqref{e7.1}-\eqref{e7.3}. By Proposition \ref{p5.6}, $V_2$ has a basis $\{Y_1,Y_2\}$ such that
\begin{eqnarray*}
&&g_2\triangleright Y_i=\zeta_3 Y_i,\  \ 1\leq i\leq 2,\\
&&g_1\triangleright Y_1=\beta_2 Y_1,\ g_1\triangleright Y_2=-\beta_2 Y_2, \\
&&g_3\triangleright Y_1=Y_{2}, \ g_3\triangleright Y_2=\gamma_2 Y_1.
\end{eqnarray*} 
Here $\beta_2, \gamma_2\in \k^*$ satisfy $\beta_2^{m_1}=\prod_{l=1}^{m_1-1}\Phi_{g_2}(g_1, g_1^l),$ 
$\gamma_2^{\frac{m_3}{2}}=\prod_{l=1}^{m_3-1}\Phi_{g_2}(g_3, g_3^l).$ So we obtain a diagonal basis $\{X_1, X_2, Y_1,Y_2\}$ of $V_1\oplus V_2$.

If $\beta_1\beta_2\neq \pm 1$, the diagram $\mathcal{D}(V_1\oplus V_2)$ is 
\[ {\setlength{\unitlength}{1.5mm}
\Rchainfour{$1$}{$\zeta_3$}{$1$}{$\zeta_3$}{$-\beta_1\beta_2$}{$\beta_1\beta_2$}{$-\beta_1\beta_2$}{$\beta_1\beta_2$}{$\zeta_2^3$}} \quad .\] 
Because it contains a $4$-cycle, we get $\GKdim(B(V_1\oplus V_2))=\infty$.

If $\beta_1\beta_2= \pm 1$, then $\mathcal{D}(V_1\oplus V_2)$ becomes
\[ {\setlength{\unitlength}{1.5mm}
\Mchainfour{$1$}{$-1$}{$\zeta_3$}{$\zeta_3^2$}{$\zeta_3$}{$-1$}{$1$}} \quad .\] 
Again, this diagram does not belong to $\mathfrak{D}$. Consequently, we have $\GKdim(B(V_1\oplus V_2))=\infty$.

(3). $V_1$ is of type (T2) and $V_2$ is of type (T3). 

By Proposition \ref{p5.6}, $V_1$ has a basis $\{X_1,X_2\}$ such that 
\begin{eqnarray*}
&&g_1\triangleright X_i= -X_i,\  \ 1\leq i\leq 2,\label{eq7.1} \\
&&g_2\triangleright X_1=\beta_1 X_1,\ g_2\triangleright X_2=-\beta_1 X_2,\label{eq7.2} \\
&&g_3\triangleright X_1=X_{2}, \ g_3\triangleright X_2=\gamma_1 X_1,\label{eq7.3}
\end{eqnarray*}
where $\beta_1, \gamma_1\in \k^*$ satisfy $\beta_1^{m_2}=\prod_{l=1}^{m_2-1}\Phi_{g_1}(g_2, g_2^l),$ $\gamma_1^{\frac{m_3}{2}}=\prod_{l=1}^{m_3-1}\Phi_{g_1}(g_3, g_3^l).$

Similarly,  $V_2$ has a basis $\{Y_1,Y_2\}$ such that 
\begin{eqnarray*}
&&g_2\triangleright Y_i=\zeta_3 Y_i,\  \ 1\leq i\leq 2,\\
&&g_1\triangleright Y_1=\beta_2 Y_1,\ g_1\triangleright Y_2=-\beta_2 Y_2, \\
&&g_3\triangleright Y_1=Y_{2}, \ g_3\triangleright Y_2=\gamma_2 Y_1,
\end{eqnarray*} 
where $\beta_2, \gamma_2\in \k^*$ satisfy $\beta_2^{m_1}=\prod_{l=1}^{m_1-1}\Phi_{g_2}(g_1, g_1^l),$ 
$\gamma_2^{\frac{m_3}{2}}=\prod_{l=1}^{m_3-1}\Phi_{g_2}(g_3, g_3^l).$

If $\beta_1\beta_2\neq \pm 1$, the diagram $\mathcal{D}(V_1\oplus V_2)$ is 
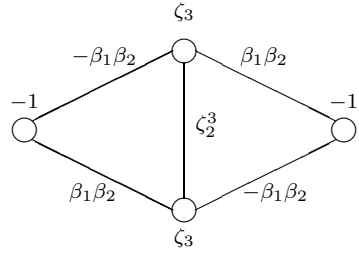
\[ {\setlength{\unitlength}{1.5mm}
\Rchainfour{$-1$}{$\zeta_3$}{$-1$}{$\zeta_3$}{$-\beta_1\beta_2$}{$\beta_1\beta_2$}{$-\beta_1\beta_2$}{$\beta_1\beta_2$}{$\zeta_2^3$}} \quad .\] 
This implies $\GKdim(B(V_1\oplus V_2))=\infty$ since $\mathcal{D}(V_1\oplus V_2)$ cotains a $4$-cycle.

If $\beta_1\beta_2= \pm 1$, the diagram $\mathcal{D}(V_1\oplus V_2)$ becomes
\[ {\setlength{\unitlength}{1.5mm}
\Mchainfour{$-1$}{$-1$}{$\zeta_3$}{$\zeta_3^2$}{$\zeta_3$}{$-1$}{$-1$}} \quad .\] 
Once more, the diagram $\mathcal{D}(V_1\oplus V_2)$ is not of finite type, therefore $\GKdim(B(V_1\oplus V_2))=\infty$.
\end{proof}

We are now ready to prove the main result of this section.

\begin{proposition}\label{p7.4}
For every $U\in \mathcal{F}_2$, we have $\GKdim(B(U))=\infty$.
\end{proposition}
\begin{proof}
Let $U\in \mathcal{F}_2$. We can decompose $U$ as a direct sum $U=U_1\oplus U_2\oplus U_3$ of simple objects. If $U\notin \mathcal{F}_2^{T}$, then for some $i\in \{1,2,3\}$, the simple object $U_i$ is not of type (T1), (T2) or (T3). Consequently, $\GKdim(B(U_i))=\infty$, this forces $\GKdim(B(U))=\infty$. On the other hand, if $U\in \mathcal{F}_2^{T}$, we also have $\GKdim(B(U))=\infty$ by Propositions \ref{p6.9} and \ref{p7.1}-\ref{p7.3}.
\end{proof}

Combining  propositions \ref{p5.9} and \ref{p7.4}, we obtain the following conclusion.

\begin{corollary}\label{c7.5}
Let $U\in {_{\k G}^{\k G} \mathcal{YD}^\Phi}$ be a minimal nondiagonal object such that $G_U=G$. Then we have $\GKdim(B(U))=\infty$.
\end{corollary}

\section{Classification}
In this section, we prove that every pre-Nichols algebra of a nondiagonal object in $_{\k G}^{\k G} \mathcal{YD}^\Phi$ has infinite GKdim. As applications, we obtain a classification of all finite-dimensional objects in $_{\k G}^{\k G} \mathcal{YD}^\Phi$ whose Nichols algebras have finite GKdim. 

Let $V\in {_{\k G}^{\k G} \mathcal{YD}^\Phi}$ be a minimal nondiagonal object. By Lemma \ref{l3.1}, $V$ is naturally an object in $_{\k G_V}^{\k G_V} \mathcal{YD}^{\Phi_{G_V}}$. It may happen, however, that $V$ is no longer a minimal nondiagonal object in the latter category when $G_V\neq G$. We need the following lemma to address this situation. 
\begin{lemma}\label{l8.1}
Let $V\in {_{\k G}^{\k G} \mathcal{YD}^\Phi}$ be a minimal nondiagonal object. Then there exists a minimal nondiagonal object $V'$ in $_{\k G_V}^{\k G_V} \mathcal{YD}^{\Phi_{G_V}}$ such that $V'\subset V$ as objects in $_{\k G_V}^{\k G_V} \mathcal{YD}^{\Phi_{G_V}}$ and $G_{V'}=G_V$.
\end{lemma}
\begin{proof}
Write $V=V_1\oplus V_2\oplus V_3$ as a direct sum of simple objects in ${_{\k G}^{\k G} \mathcal{YD}^\Phi}$, and set $g_i=g_{V_i}$ for $1\leq i\leq 3$. Then $G_V=\langle g_1,g_2,g_3\rangle$. Because $V$ is nondiagonal, Lemma \ref{l2.4} implies that $\Phi_{G_V}$ is a nonabelian $3$-cocycle on $G_V$. For each $i\in \{1,2,3\}$, choose a simple object $V'_i\subset V_i$ in the category $_{\k G_V}^{\k G_V} \mathcal{YD}^{\Phi_{G_V}}$ and let $V'=V'_1\oplus V'_2\oplus V'_3$. It is clear that $G_{V'}=G_V$. Since $V'$ is a direct sum of $3$-simple objects and $\Phi_{G_V}$ is a nonabelian $3$-cocycle on $G_V=G_{V'}$, Proposition \ref{p5.4} shows that $V'$ is a minimal nondiagonal object in $_{\k G_V}^{\k G_V} \mathcal{YD}^{\Phi_{G_V}}$. 
\end{proof}

\begin{proposition}\label{p8.2}
Let $V\in {_{\k G}^{\k G} \mathcal{YD}^\Phi}$ be a minimal nondiagonal object. Then every pre-Nichols algebra of $V$ has infinite GKdim.
\end{proposition}
\begin{proof}
Because there exists a canonical projection from every pre-Nichols algebra $\mathcal{P}(V)$ onto $B(V)$, it suffices to show that $\GKdim(B(V))=\infty$.
By Lemma \ref{l3.1}, $B(V)$ is also a Nichols algebra in $_{\k G_V}^{\k G_V} \mathcal{YD}^{\Phi_{G_V}}$. According to Lemma \ref{l8.1}, there exists a minimal nondiagonal object $V'\subset V$ in ${_{\k G_V}^{\k G_V} \mathcal{YD}^{\Phi_{G_V}}}$ with $G_{V'}=G_V$. By Corollary \ref{c7.5}, we have $\GKdim(B(V'))=\infty$. Since $B(V')$ is an $\N_0$-graded subalgebra of $B(V)$, we conclude that $\GKdim(B(V))=\infty$.
\end{proof}
We are now ready to state our main result.
\begin{theorem}\label{t8.3}
Let $V\in\ _{\k G}^{\k G} \mathcal{YD}^\Phi$ be a nondiagonal object. Then every pre-Nichols algebra of $V$ has infinite GKdim.
\end{theorem}
\begin{proof}
It suffices to show that $\GKdim(B(V))=\infty$. Because $V\in {_{\k G}^{\k G} \mathcal{YD}^\Phi}$ is nondiagonal, Corollary \ref{cor5.5} guarantees the existence of a minimal nondiagonal subobject $V'\subset V$. By Proposition \ref{p8.2}, we have  $\GKdim(B(V'))=\infty$. Consequently, we obtain $\GKdim(B(V))=\infty$.
\end{proof}

As an immediate consequence, we classify all finite GK-dimensional Nichols algebras in $_{\k G}^{\k G} \mathcal{YD}^\Phi$. 

\begin{theorem}\label{t8.4}
Let $V\in{_{\k G}^{\k G} \mathcal{YD}^\Phi}$ be a finite-dimensional object. Then $B(V)$ has finite GKdim if and only if it is of diagonal type and the corresponding root system is finite, namely, an arithmetic root system.
\end{theorem}
\begin{proof}
Assume $\GKdim(B(V))<\infty$. By Theorem \ref{t8.3}, $B(V)$ must be of diagonal type. Then Theorem \ref{t4.2} forces the root system $\D(B(V))$ to be finite. Conversely, if $V$ is diagonal and the corresponding root system is finite, by Theorem \ref{t4.2} again we obtain $\GKdim(B(V))<\infty$.
\end{proof}

This classification leads to a structural description of finite GK-dimensional coradically graded pointed coquasi-Hopf algebras over finite abelian groups that are generated by group-like and skew-primitive elements.
\begin{corollary}\label{c8.5}
Let $H$ be a coradically graded pointed coquasi-Hopf algebra over $G$ generated by group-like and skew-primitive elements. If $H$ is finitely generated and $\GKdim(H)<\infty$, then $$H\cong B(V)\# \k G$$ for some $V\in{_{\k G}^{\k G} \mathcal{YD}^\Phi}$ whose    generalized Dynkin diagram $\mathcal{D}(V)$ is of finite type, where $\Phi$ is a $3$-cocycle on $G$ determined by the associator of $H$.
\end{corollary}
\begin{proof}
Because $H$ is coradically graded, it decomposes as a bosonization  $H\cong R\# \k G$, where $R=\oplus_{i\geq 0}R_i$ is a coradically graded connected braided Hopf algebra in ${_{\k G}^{\k G} {\mathcal{YD}^\Phi}}$, and $\Phi$ is the $3$-cocycle on $G$ determined by the associator of $H$.  Let $P(R)$ denote the subspace of primitive elements of $R$. Since $R$ is coradically graded, we have $P(R)=R_1$. Furthermore, the assumption that $H$ is generated by group-like and skew-primitive elements implies that $R$ is generated by $P(R)=R_1$. Therefore, $R$ is a pre-Nichols algebra of $R_1$ such that $P(R)=R_1$, this forces $R$ to be the Nichols algebra $B(R_1)$.

Now let $V=R_1$, we obtain $H\cong B(V)\# \k G$. If $H$ is finitely generated and $\GKdim(H)<\infty$, then $V$ must be finite dimensional and $\GKdim(B(V))<\infty$. By Theorem \ref{t8.4}, $B(V)$ is of diagonal type and its root system is finite, hence the generalized Dynkin diagram $\mathcal{D}(V)$ is of finite type.
\end{proof}

 \section*{acknowledgments}
The author was supported by the National Natural Science Foundation of China (No. 12471020), the Natural Science Foundation of Chongqing (No. cstc2021 jcyj-msxmX0714), and the Fundamental Research Funds for the Central Universities (No. SWU-XDJH202305, SWU-KT25015). 




\end{document}